\documentclass[10pt,oneside,leqno]{amsart}
\usepackage{amsxtra}
\usepackage{amsopn}
\usepackage{color}
\usepackage{amsmath,amsthm,amssymb}
\usepackage{amscd}
\usepackage{amsfonts}
\usepackage{latexsym}
\usepackage{verbatim}
\usepackage{xypic}
\newtheorem{definition}{Definition}[section]
\newtheorem{lemma}[definition]{Lemma}
\newtheorem{theorem}[definition]{Theorem}
\newtheorem{proposition}[definition]{Proposition}
\newtheorem{corollary}[definition]{Corollary}
\newtheorem{remark}[definition]{Remark}

\newtheorem{example}[definition]{Example}

 1  %Caracteres "doble palo".
\def\QED{\hskip0.1em\hfill\null\ \null\nobreak\hfill
\kern3pt\lower1.8pt\vbox{\hrule\hbox {\vrule\kern1pt\vbox{\kern1.7pt
\hbox{$\scriptstyle QED$}\kern0.2pt}\kern1pt\vrule}\hrule}}
\def\fra{{\mathfrak C}}
\def\frg{{\frak g}}

\def\hr{{\rm hr}}

\newcommand{\hb}[1]{\hbox{#1}}
\newcommand{\SSS}{{\mathcal S}}

\newcommand{\M}{(M^{2n},\omega)}

\begin{document}

\title[]{\bf Symplectic harmonicity and generalized coeffective cohomologies}

\author{\bf Luis Ugarte}
\address[L. Ugarte]{Departamento de Matem\'aticas, IUMA, Universidad de Zaragoza\\
Campus Plaza San Francisco, 50009 Zaragoza, Spain}
\email{ugarte@unizar.es}

\author{\bf Raquel Villacampa}
\address[R. Villacampa]{Centro Universitario de la Defensa Zaragoza\\
IUMA, Universidad de Zaragoza\\
Academia General Militar,
Crta. de Huesca s/n, 50090 Zaragoza, Spain}
\email{raquelvg@unizar.es}

%%%\date{\today}

\maketitle

\bigskip

{\small {\noindent {\sf Abstract}.-  Relations between the symplectically harmonic cohomology and the coeffective cohomology
of a symplectic manifold are obtained.
This is achieved through a generalization of the latter, which in addition allows us to provide a
coeffective version of the filtered cohomologies introduced by C.-J. Tsai, L.-S. Tseng and S.-T. Yau.
We construct closed (simply connected) manifolds endowed with a family of symplectic forms $\omega_t$
such that the dimensions of these symplectic cohomology groups vary with respect to $t$.
A complete study of these cohomologies is given for 6-dimensional symplectic nilmanifolds, and concrete examples
with special cohomological properties are obtained on an $8$-dimensional solvmanifold and on 2-step nilmanifolds in higher dimensions.
} }

\medskip

{\small \noindent{\sf Keywords:} {\it Symplectic Hodge theory, coeffective cohomology, filtered and primitive cohomologies,
Lefschetz map.}}

%%%\medskip

%%%\noindent{\sf 2010 MSC.} Primary 53D05; Secondary 58J10, 53D35, 58D35}

\bigskip

%%%%%%%%%%%%%%%%%%%%%%%%%%%%%%%%%%%%%%%%%%%%%%%%%%%%%%%%%%%%%%%%%%%%%%%%%%%%%%%%%%%%%%%%%%%%%%%%%%%%%%%%%%%%%%%%%%%
%%%%%%%%%%%%%%%%%%%%%%%%%%%%%%%%%%%%%%%%%%%%%%%%%%%%%%%%%%%%%%%%%%%%%%%%%%%%%%%%%%%%%%%%%%%%%%%%%%%%%%%%%%%%%%%%%%%
\section{Introduction}
%%%%%%%%%%%%%%%%%%%%%%%%%%%%%%%%%%%%%%%%%%%%%%%%%%%%%%%%%%%%%%%%%%%%%%%%%%%%%%%%%%%%%%%%%%%%%%%%%%%%%%%%%%%%%%%%%%%
%%%%%%%%%%%%%%%%%%%%%%%%%%%%%%%%%%%%%%%%%%%%%%%%%%%%%%%%%%%%%%%%%%%%%%%%%%%%%%%%%%%%%%%%%%%%%%%%%%%%%%%%%%%%%%%%%%%

\noindent
Let $\M$ be a symplectic manifold.
The notion of \emph{symplectically harmonic form} was introduced by Brylinski in \cite{Br} as
a closed form $\alpha$ such that its symplectic star is also closed, i.e. $d\alpha=0=d*\alpha$.
Mathieu~\cite{Ma} proved (see also \cite{Yan} for a different proof)
that every de Rham cohomology class has a symplectically harmonic
representative if and only if $\M$ satisfies the Hard Lefschetz Condition (HLC for short), i.e.
the homomorphisms $L^k\colon H^{n-k}(M)\longrightarrow H^{n+k}(M)$ are surjective for every
$1\leq k\leq n$. Here $H^q(M)$ denotes the $q$-th de Rham cohomology group of $M$ and $L^k$ is the homomorphism
given by the cup product with the class $[\omega^k]\in H^{2k}(M)$.
Since there exist many symplectic manifolds which do not satisfy the HLC, one has that the quotient
%%%$H^q_\hr(M)=\frac{\Omega^q_\hr(M)}{\Omega^q_\hr(M)\cap \text{im}\, d}$,
$H^q_\hr(M)=\Omega^q_\hr(M)/(\Omega^q_\hr(M)\cap \text{im}\, d)$,
$\Omega^q_\hr(M)$ being the space of symplectically
harmonic $q$-forms, counts the de Rham cohomology classes in $H^q(M)$ containing harmonic representative.

Additional symplectic invariants of cohomological type were introduced by Bouch\'e \cite{B} as follows.
A differential form $\alpha$ is called \emph{coeffective} if it annihilates $\omega$, i.e. $\alpha\wedge\omega=0$.
The space of coeffective forms with the (restriction of the) exterior derivative provides
a subcomplex of the de Rham complex
that is elliptic in any degree $q\not= n$.
It turns out~\cite{B} that for compact K\"ahler manifolds $\M$ and for every $q\geq n+1$,
the $q$-th coeffective cohomology group, that we will denote here by $H^q_{(1)}(M)$,
is isomorphic to the \emph{$[\omega]$-truncated $q$-th de Rham group}.
However, this is no longer true for arbitrary compact symplectic manifolds~\cite{FIL1}.
%%%whereas in \cite{FIL2} relations between the coeffective numbers and the topology of the symplectic manifold are shown.
%%%
On the other hand, Tseng and Yau have developed in \cite{TY1,TY2} a symplectic Hodge theory by
considering various cohomologies where the \emph{primitive} cohomologies $PH_{d+d^\Lambda}(M), PH_{dd^\Lambda}(M),
PH_{\partial_+}(M)$ and $PH_{\partial_-}(M)$ play a central role.
Recently, Eastwood~\cite{E} has introduced an extension of the coeffective
complex which is elliptic in any degree and such that the corresponding cohomology groups are isomorphic
to the primitive cohomology groups.

The symplectically harmonic cohomology and the coeffective cohomology, to our knowledge, have been studied separately in the literature.
%%%, to our knowledge, no link between them has been given. Our goal in this paper is
Our first goal in this paper is
to obtain some relations between both cohomologies by considering a
natural generalization of the coeffective cohomology, which in addition will allow us to provide a
coeffective version of the \emph{filtered} cohomologies.
The latter have recently been introduced by Tsai, Tseng and Yau \cite{TY3},
and extend the primitive cohomologies \cite{TY1,TY2}.
%%%Since the generalized coeffective complex is just a subcomplex of de Rham, so its cohomology is
%%%in general simpler to consider than other symplectic cohomology groups. This is the main objective
%%%of the paper, to related other more complicated symplectic cohomological invariants to the generalized
%%%coeffective cohomologies, so that one can derive some properties of the former by studying the latter.

Another aspect in the study of the symplectic harmonicity is the notion of \emph{flexibility}, %%%\cite{IRTU1,Yan}. This is 
motivated by the following question, which seems to be related to some problems of
group-theoretical hydrodynamics~\cite{AK}, posed by Khesin and McDuff (see \cite{Yan}):
which closed manifolds $M$ possess a continuous family $\omega_t$ of symplectic forms such
that the dimension of $H^q_\hr(M,\omega_t)$ varies with respect to $t$?
In~\cite{Yan} Yan proved the existence of a 4-dimensional flexible manifold, whereas in \cite{IRTU1}
several 6-dimensional nilmanifolds satisfying such property were found.
Recently, Cho has proved in \cite{Cho} the existence of simply-connected flexible examples of dimension six.
Our second goal in this paper is to relate the harmonic flexibility to corresponding notions
of flexibility for the generalized coeffective and filtered cohomologies,
as well as to construct closed manifolds which are flexible
with respect to these symplectic cohomologies.

In greater detail, the paper is structured as follows.

In Section~\ref{gen-coef-cohom}
we introduce and study the generalized coeffective cohomologies of a symplectic manifold $\M$.
For each integer $k$, $1\leq k\leq n$, we consider the complex of $k$-coeffective differential forms
as the subcomplex of de Rham one constituted by all the forms that annihilate $\omega^k$.
The associated cohomology groups are denoted by $H^q_{(k)}(M)$.
This complex is elliptic in any degree $q\not= n-k+1$, however one can define in a natural way
a quotient $\hat H^{n-k+1}(M)$ of $H^{n-k+1}_{(k)}(M)$ which shares the same properties as the cohomology groups
$H^q_{(k)}(M)$, $q\geq n-k+2$
(see Propositions~\ref{propiedades-finitas} and~\ref{propiedades-finitas-para-c-sombrero}). The spaces
$\hat H^{1}(M), \ldots, \hat H^{n}(M)$
play an important role in this paper since they will allow us to relate the
different symplectic cohomologies involved.
We will refer to the collection
\begin{equation}\label{todos-los-gen-coef-groups}
\hat H^{n-k+1}(M),\ H^{n-k+2}_{(k)}(M),\ldots,\ H^{2n}_{(k)}(M),\quad\ 1\leq k\leq n,
\end{equation}
as the \emph{generalized coeffective} cohomology groups of the symplectic manifold $\M$.
It turns out that these spaces are symplectic invariants that only depend on the
de Rham class $[\omega^k] \in H^{2k}(M)$
(see Remark~\ref{coef-inv-cohom-class} and Lemma~\ref{iso-equiv-coef}).
When $M$ is of finite type, in Proposition~\ref{invtopologicos} we prove that, for each $1\leq k\leq n$, the alternating
sum $\chi^{(k)}(M)$ of the dimensions of the generalized $k$-coeffective cohomology groups
only depends on the topology of the manifold $M$.

Eastwood~\cite{E} has introduced an elliptic extension of the usual coeffective
complex (i.e. $k=1$) such that the corresponding cohomology groups are isomorphic to primitive cohomology groups
defined by Tseng and Yau~\cite{TY1,TY2}.
In Section~\ref{exten-gen-coef-cohom},
for any $1\leq k\leq n$, we consider an extension of the $k$-coeffective complex,
which is also elliptic,
whose cohomology groups $\check{H}^q_{(k)}(M)$ $(0\leq q\leq 2n+2k-1)$
are isomorphic to the filtered cohomology groups introduced by Tsai, Tseng and Yau in \cite{TY3}
(see Remark~\ref{comparison with filtered} %%%and~\eqref{comparar-con-primitiva}
for details); in particular,
$$
\check{H}^{q}_{(1)}(M) \cong P H^q_{\partial_+}(M), \quad\
\check{H}^{2n-q+1}_{(1)}(M) \cong H^{2n-q}_{(1)}(M) \cong P H^q_{\partial_-}(M),
\quad\ 0\leq q\leq n-1,
$$
and
$$
\check{H}^{n+k-1}_{(k)}(M) \cong P H^{n-k+1}_{dd^\Lambda}(M), \quad\
\check{H}^{n+k}_{(k)}(M) \cong P H^{n-k+1}_{d+d^\Lambda}(M), \quad\ 1\leq k\leq n.
$$
In Proposition~\ref{propiedades-finitas-suc-check} we show that these extended cohomologies
also satisfy the main properties of the generalized coeffective cohomology groups.
%%%(this is mainly due to the existence
%%%of a long exact sequence related to the homomorphisms $L^k\colon H^q(M)\longrightarrow H^{q+2k}(M)$).
When $M$ is of finite type, we consider $\check{\chi}^{(k)}_+(M)$ as the alternating sum of
the dimensions of the cohomology groups of the first half
of the extended complex, and in
Corollary~\ref{HLC-caracterizacion} we prove the following characterization of the HLC:
$\M$ satisfies the HLC if and only if
$\check{\chi}^{(k)}_+(M)=\chi^{(k)}(M)$ for every $1\leq k\leq n$.

In Section~\ref{relacion-con-armonica} %4
we obtain some relations of the generalized coeffective cohomologies (and therefore also of the
filtered cohomologies) with the symplectically harmonic cohomology.
Concretely, using the description of $H^q_\hr(M)$ obtained in \cite{IRTU1,Ym,Yan}
we prove that the generalized coeffective cohomologies measure the differences between
the harmonic cohomology groups in the following sense:
if $\M$ is a symplectic manifold of finite type, then
$$
\dim H^{n-k+1}_\hr(M) - \dim H^{n+k+1}_\hr(M) = \dim \hat H^{n-k+1}(M)
$$
for every $k=1,\ldots,n$ (see Theorem~\ref{rel_armonica_coef}).
As a consequence, we find
%%%in Corollary~\ref{rel_armonica_coef-2-cor}
the relation between
the dimension of the primitive cohomology group
$P H^q_{d+d^\Lambda}(M)$ and the harmonic cohomology for $q=1,2,3$.

We introduce in Section~\ref{flex} %5
the notion of generalized coeffective flexibility and filtered flexibility,
as an analogous notion of the concept of harmonic flexibility.
%%%introduced and studied in \cite{Yan,IRTU1}
%%%and motivated by a question raised by Boris Khesin and Dusa McDuff (see \cite{Yan}).
We say that a closed smooth manifold $M^{2n}$ is \texttt{c}-\emph{flexible} (resp. \texttt{f}-\emph{flexible} or \texttt{h}-\emph{flexible})
if $M$ possesses a continuous family of symplectic forms $\omega_t$
such that the dimension of some generalized coeffective (resp. filtered or symplectically harmonic) cohomology group varies with $t$.
We prove in Theorem~\ref{dim4} that in four dimensions $M$ is never \emph{\texttt{c}}-flexible, and that
$M$ is \emph{\texttt{f}}-flexible if and only if it is \emph{\texttt{h}}-flexible.
This result allows us to prove, for each $n\geq 2$, the existence of $2n$-dimensional \emph{\texttt{f}}-flexible closed manifolds
having a continuous family of symplectic forms $\omega_t$
such that the dimension of the primitive cohomology group $P H^{2}_{d+d^\Lambda}(M,\omega_t)$
%%%and $P H^{2}_{dd^\Lambda}(M,\omega_t)$
varies with respect to $t$ (see Theorem~\ref{existencia-f-flexible-en-cualquier-dimension}).
In Theorem~\ref{existencia-f-flexible-en-cualquier-dimension-simply-connected} we use a result in \cite{Cho}
to prove that,
for every $n\geq 3$, there exists a $2n$-dimensional simply-connected closed manifold $M$ with a continuous family $\omega_t$
for which the dimensions of the primitive groups $P H^{3}_{d+d^\Lambda}(M,\omega_t)$ and $P H^{3}_{dd^\Lambda}(M,\omega_t)$
vary with~$t$.
In Theorem~\ref{dim6} and Proposition~\ref{dim2n} we study flexibility in higher dimensions; in particular, it turns out that
in dimension $2n\geq 6$,
if $M$ is \emph{\texttt{c}}-flexible then $M$ is \emph{\texttt{f}}-flexible or \emph{\texttt{h}}-flexible.
This shows that coeffective flexibility is a stronger condition than the other flexibilities.

All the cohomology groups can be computed explicitly for symplectic solvmanifolds satisfying the Mostow condition,
in particular for any symplectic nilmanifold.
In Section~\ref{sec_ex} %6
we consider the class of 6-dimensional nilmanifolds and compute the dimensions of all the cohomology groups for
any symplectic form. This extends the previous study for the symplectically harmonic cohomology given in \cite{IRTU1,IRTU2}.
As a consequence, we identify all the 6-dimensional nilmanifolds which are \emph{\texttt{c}}-flexible, \emph{\texttt{f}}-flexible or \emph{\texttt{h}}-flexible (see Table~1).
A solvmanifold of dimension 8
that is \emph{\texttt{c}}-flexible, \emph{\texttt{f}}-flexible and \emph{\texttt{h}}-flexible
is described in Section~\ref{8-dim-solvmanifold}.
{Section~\ref{Sakane-Yamada-sec} is devoted to symplectic $2$-step nilmanifolds and, based on
results by Sakane and Yamada \cite{SY-proc,Ym}, we obtain examples of arbitrary high dimension
which are \emph{\texttt{c}}-flexible, \emph{\texttt{f}}-flexible and \emph{\texttt{h}}-flexible.

%%%%%%%%%%%%%%%%%%%%%%%%%%%%%%%%%%%%%%%%%%%%%%%%%%%%%%%%%%%%%%%%%%%%%%%%%%%%%%%%%%%%%%%%%%%%%%%%%%%%%%%%%%%%%%%%%%%
%%%%%%%%%%%%%%%%%%%%%%%%%%%%%%%%%%%%%%%%%%%%%%%%%%%%%%%%%%%%%%%%%%%%%%%%%%%%%%%%%%%%%%%%%%%%%%%%%%%%%%%%%%%%%%%%%%%%
\section{Generalized coeffective cohomologies}\label{gen-coef-cohom}
%%%%%%%%%%%%%%%%%%%%%%%%%%%%%%%%%%%%%%%%%%%%%%%%%%%%%%%%%%%%%%%%%%%%%%%%%%%%%%%%%%%%%%%%%%%%%%%%%%%%%%%%%%%%%%%%%%%
%%%%%%%%%%%%%%%%%%%%%%%%%%%%%%%%%%%%%%%%%%%%%%%%%%%%%%%%%%%%%%%%%%%%%%%%%%%%%%%%%%%%%%%%%%%%%%%%%%%%%%%%%%%%%%%%%%%%

\noindent Let $\M$ be a symplectic manifold of dimension $2n$ and let $k$ be an integer such that $1\leq k\leq n$.
Next we introduce the notion of $k$-coeffective forms.

\begin{definition}\label{def-gen-coef}
{\rm
A $q$-form $\alpha$ on $M$ is said to be \emph{$k$-coeffective} if $\alpha$ annihilates the form $\omega^k$,
i.e. $\alpha\wedge \omega^k=0$.
The space of $k$-coeffective forms of degree $q$ will be denoted by $\fra^q_{(k)}(M,\omega)$, or simply $\fra^q_{(k)}(M)$.
}
\end{definition}

\begin{remark}\label{limit-cases}
{\rm
The above definition makes also sense in the ``limit'' case $k=n+1$ because $\omega^{n+1}=0$ and then $\fra^*_{(n+1)}(M)=\Omega^*(M)$.
Also the case $k=0$ makes sense if we consider
$\omega^0$ as the constant function 1, i.e. $\fra^*_{(0)}(M)=\{0\}$.
Thus, there exists the following strictly increasing sequence
of differential ideals
$$
\{0\}=\fra^*_{(0)}(M) \subset \fra^*_{(1)}(M) \subset \cdots \subset \fra^*_{(n)}(M)
\subset \fra^*_{(n+1)}(M) =\Omega^*(M).
$$
}
\end{remark}

Since for each $k$ the space $\fra^*_{(k)}(M)$ is closed by $d$, we can consider the \emph{$k$-coeffective complex}
\begin{equation}\label{complejo-k-coef}
\xymatrix{\cdots\ar [r]^-d &\fra^{q-1}_{(k)}(M)\ar[r]^d
&\fra^{q}_{(k)}(M)\ar[r]^d &\fra^{q+1}_{(k)}(M)\ar[r]^-d&\cdots},
\end{equation}
which is a subcomplex of the standard de Rham complex $(\Omega^*(M),d)$.
%%%From now on, we denote the $q$-th de Rham cohomology group of $M$ by $H^q(M)$.

\begin{definition}\label{def-gen-coef-cohom}
{\rm
The \emph{$q$-th $k$-coeffective cohomology group} will be denoted by
$$H^q_{(k)}(M)=\frac{\ker\,\{d:\fra^q_{(k)}(M)\longrightarrow \fra^{q+1}_{(k)}(M)\}}
{\textrm{im}\,\{d:\fra^{q-1}_{(k)}(M)\longrightarrow \fra^{q}_{(k)}(M)\}}.
$$
%%%We will also use the more simplified notation $H^q_{(k)}(M)$ if there is no confusion
%%%about the symplectic form.
}
\end{definition}

It is clear that the $k$-coeffective cohomology groups are invariant by symplectomorphism.
Moreover, we will show below that, for each $k$, they are invariants of the de Rham class $[\omega^k]\in H^{2k}(M)$.

Let $L_{\omega}^k\colon\Omega^*(M)\longrightarrow \Omega^{*}(M)$ be given by $L_{\omega}^k(\alpha)=\alpha\wedge \omega^{k}$.
Since $\fra^q_{(k)}(M) =\ker \{L_{\omega}^k:\Omega^q(M)\longrightarrow \Omega^{q+2k}(M)\}$ and
the map $L_{\omega}^k:\Omega^q(M)\longrightarrow\Omega^{q+2k}(M)$ is injective
for any $q\leq n-k$ and surjective for any $q\geq n-k$, one has that
%%%$\fra^q_{(k)}(M)=\{0\}$ for any $q\leq n-k$, which implies that
$H^q_{(k)}(M)=0$ for $q\leq n-k$ and $H^q_{(k)}(M)\cong H^q(M)$ for every $q\geq 2n-2k+2$.

The short exact sequence
$$\xymatrix{0 \ar[r]
&\fra^*_{(k)}(M)\ar[r]^-i &\Omega^*(M)\ar[r]^-{L_{\omega}^k}
&L_{\omega}^k(\Omega^{*}(M)) \ar[r] &0,}$$
where $i$ denotes the inclusion,
provides the following short exact sequence of complexes
%\medskip
\xymatrixcolsep{2pc}\xymatrixrowsep{1.5pc} $$\xymatrix{
&0 &0 &0 \\
\cdots\ar[r]^-d &L_{\omega}^k(\Omega^{q-1}(M))\ar[r]^d\ar[u]
&L_{\omega}^k(\Omega^{q}(M))\ar[r]^d\ar[u] &L_{\omega}^k(\Omega^{q+1}(M))\ar[u]\ar[r]^-d &\cdots  \\
\cdots\ar[r]^-d &\Omega^{q-1}(M)\ar[r]^d\ar[u]^{L_{\omega}^k}
&\Omega^{q}(M)\ar[r]^d\ar[u]^{L_{\omega}^k}
&\Omega^{q+1}(M)\ar[u]^{L_{\omega}^k}\ar[r]^d &\cdots\\
\cdots\ar[r]^-d&\fra^{q-1}_{(k)}(M)\ar[r]^d\ar[u]^-i
&\fra^q_{(k)}(M)\ar[r]^d\ar[u]^-i
&\fra^{q+1}_{(k)}(M)\ar[u]^-i\ar[r]^d &\cdots \\
&0\ar[u] &0\ar[u] &0\ar[u] }$$

Now, since $L_{\omega}^k(\Omega^{q-2k}(M))=\Omega^q(M)$
for any $q\geq n+k$ we have that $H^{q}(L_{\omega}^k(\Omega^{*}(M)))=H^q(M)$ for $q\geq n+k+1$, and
therefore the associated long exact sequence in cohomology is
%%%
%%%\xymatrixcolsep{2pc}\xymatrixrowsep{0.8pc}
%%%\begin{equation}\label{selarga}
%%%\xymatrix{
%%%0\ar[r]& H^{n-k}(M)\ar[r]^-{L^k} &H^{n+k}(L_{\omega}^k(\Omega^{*}(M))) \ar[r]^-{f_{n-k+1}} &H^{n-k+1}_{(k)}(M)&\\
%%%\ar[r]^-{H(i)}& H^{n-k+1}(M)\ar[r]^-{L^k} &H^{n+k+1}(M)\ar[r]^-{f_{n-k+2}}& H^{n-k+2}_{(k)}(M)&\\
%%%\ar[r]^-{H(i)}& H^{n-k+2}(M)\ar[r]^-{L^k} &H^{n+k+2}(M)\ar[r]^-{f_{n-k+3}}&H^{n-k+3}_{(k)}(M)\ \cdots\, ,&}
%%%\end{equation}
%%%
\xymatrixcolsep{3pc}\xymatrixrowsep{0.8pc}
$$\xymatrix{
0\ar[r]& H^{n-k}(M)\ar[r]^-{L^k} &H^{n+k}(L_{\omega}^k(\Omega^{*}(M))) \ar[r]^-{f_{n-k+1}} &H^{n-k+1}_{(k)}(M)&}$$
\begin{equation}\label{selarga}
\xymatrixcolsep{3.5pc}\xymatrix{
\ar[r]^-{H(i)}& H^{n-k+1}(M)\ar[r]^-{L^k} &H^{n+k+1}(M)\ar[r]^-{f_{n-k+2}}& H^{n-k+2}_{(k)}(M)&}
\end{equation}
\xymatrixcolsep{3pc}$$
\xymatrix{\ar[r]^-{H(i)}& H^{n-k+2}(M)\ar[r]^-{L^k} &H^{n+k+2}(M)\ar[r]^-{f_{n-k+3}}&H^{n-k+3}_{(k)}(M)\ \cdots\, ,&}
$$
where $H(i)$ and $L^k$ are the homomorphisms in cohomology naturally induced by $i$ and $L_{\omega}^k$, respectively,
and $f_{q}$
%%%$f_{q}\colon H^{q+2k-1}(M)\longrightarrow H^{q}_{(k)}(M)$
is the connecting homomorphism. Recall that
$f_{q}$ is defined by $f_{q}([\alpha])=[d\beta]$, where $\beta\in\Omega^{q-1}(M)$ is any $(q-1)$-form satisfying
$L_{\omega}^k(\beta)=\alpha$.

\begin{definition}\label{def-gen-coef-numbers}
{\rm
When the group $H^q_{(k)}(M)$ has finite dimension, we will
denote it by $c_q^{(k)}(M)$
and we shall refer to it as the \emph{$q$-th $k$-coeffective number} of $\M$.
}
\end{definition}

Notice that $c_q^{(k)}(M)=0$ for any $q\leq n-k$, and $c_q^{(k)}(M)= b_q(M)$ for any $q\geq 2n-2k+2$.

\medskip

In what follows, by a manifold of \emph{finite type} we mean a manifold, not necessarily compact,
such that its Betti numbers $b_q(M)=\dim H^q(M)$ are all finite.

\begin{proposition}\label{propiedades-finitas}
Let $\M$ be a symplectic manifold of finite type and let $1\leq k\leq n$.
Then, for every $q\geq n-k+2$, the following properties hold:
\begin{enumerate}
\item[{\rm (i)}] \emph{Finiteness and bounds for the coeffective numbers:} the group
$H^q_{(k)}(M)$ is finite dimensional and its dimension $c_q^{(k)}(M)$ satisfies
the inequalities
\begin{equation}\label{des}
b_q(M)-b_{q+2k}(M)\leq c_q^{(k)}(M)\leq b_q(M)+b_{q+2k-1}(M).
\end{equation}
\item[{\rm (ii)}] \emph{Symplectic manifolds satisfying the HLC:} if $\M$ satisfies the HLC then
the lower bound in \eqref{des} is attained, i.e.
$$
c_q^{(k)}(M)=b_q(M) - b_{q+2k}(M).
$$
\item[{\rm (iii)}] \emph{Exact symplectic manifolds:} if $\omega$ is exact then
the upper bound in \eqref{des} is attained, i.e.
$$
c_q^{(k)}(M)=b_q(M)+b_{q+2k-1}(M).
$$
\item[{\rm (iv)}] \emph{Poincar\'e lemma:} if $U$ is the open unit disk in $\mathbb{R}^{2n}$ with the standard
symplectic form $\omega=\sum_{i=1}^n dx^i\wedge dx^{n+i}$, then
$c_q^{(k)}(U)=0$.
\end{enumerate}
\end{proposition}

\begin{proof}
From the long exact sequence \eqref{selarga}, one has for any $q\geq n-k+2$ the five-term exact sequence
\xymatrixcolsep{1.6pc}
\begin{equation}\label{5terminos}
\xymatrix{0\ar[r]&\text{im}\,f_{q}\, \hookrightarrow\, H^q_{(k)}(M)\ar[r]^-{H(i)}&
H^q(M)\ar[r]^-{L^k}&H^{q+2k}(M)\ar[r]^-{f_{q+1}\,\,}&\,\text{im}\,f_{q+1}\ar[r]&0}.
\end{equation}
If the manifold is of finite type then it is clear that $H^q_{(k)}(M)$ has finite dimension for any $q\geq n-k+2$.
Moreover, taking dimensions in \eqref{5terminos}
%%%\begin{equation}\label{dim-5-term}
$$
\xymatrix{c_q^{(k)}(M)
%%%=\dim H^q_{(k)}(M)
=\text{dim (im}\,f_{q})+b_q(M) -b_{q+2k}(M)+\text{dim (im}\,f_{q+1})}\!,
%%%\end{equation}
$$
which implies the inequalities \eqref{des}. This completes the proof of {\rm (i)}.

Property {\rm (ii)} is a direct consequence of \eqref{5terminos} taking into account that HLC implies that
$L^k\colon H^{q-1}(M)\longrightarrow H^{q+2k-1}(M)$ are surjective and then the connecting homomorphisms
$f_{q}$ vanish for every $q\geq n-k+2$.

Property {\rm (iii)} is a consequence of \eqref{5terminos} since
$L^k\colon H^{q-1}(M)\longrightarrow H^{q+2k-1}(M)$ are identically zero because~$\omega$ is exact,
and then the connecting homomorphisms $f_{q}$ are injective for every $q\geq n-k+2$.

Finally {\rm (iv)} is a direct consequence of {\rm (iii)} since $\omega$ is exact on $U$.
\end{proof}

Notice that for $k=1$ the previous proposition was proved by Fern\'andez, Ib\'a\~nez and de Le\'on in \cite{FIL2}.
It is easy to check (see \cite{B} for $k=1$) that, for each $1\leq k\leq n$, the $k$-coeffective complex
\eqref{complejo-k-coef} is elliptic in any degree $q\not= n-k+1$.
The coeffective group $H^{n-k+1}_{(k)}(M)$ can be infinite dimensional, however in view of the sequence \eqref{selarga}
there is a natural quotient of this coeffective group by considering
the (in general also infinite dimensional) space $H^{n+k}(L_{\omega}^k(\Omega^{*}(M)))$.
We will see below that such quotient has finite dimension on symplectic manifolds of finite type.

\begin{definition}
%%%For every $1\leq k\leq n$, let us consider the space
Let us consider the space
$$\hat H^{n-k+1}(M):=
\frac{H^{n-k+1}_{(k)}(M)}{H^{n+k}(L_{\omega}^k(\Omega^{*}(M)))/H^{n-k}(M)}.$$
If its dimension is finite, then we will denote it by
$\hat c_{n-k+1}(M)$.
%%%We will also denote this space by $\hat H^{n-k+1}_{(k)}(M)$ and its dimension by $\hat c_{n-k+1}(M)$
%%%when the symplectic form is clear from the context.
\end{definition}

Hence, we have an additional collection of $n$ symplectic invariants given by $\hat H^{n-k+1}(M)$ for $k=1,\ldots,n$,
that is,
$$\hat H^{n}(M),\  \hat H^{n-1}(M), \ldots, \hat H^{2}(M),\ \hat H^{1}(M).$$
From now on, we will refer to the collection \eqref{todos-los-gen-coef-groups}
as the \emph{generalized coeffective cohomology groups} of the symplectic manifold $\M$.

\medskip

Consider the short exact sequence
\xymatrixcolsep{2.35pc}
$$\xymatrix{
0\ar[r]&\hat H^{n-k+1}(M)\ar[r]^{\hat{\!\emph{\i}}}&H^{n-k+1}(M)\ar[r]^-{L^k}&
H^{n+k+1}(M)\ar[r]^-{f_{n-k+2}\,\,}& \text{im}\,f_{n-k+2}\ar[r]&0},
$$
where $\hat{\!\emph{\i}}$ is the homomorphism naturally induced by $H(i)$ in \eqref{selarga}.
Since $\hat{\!\emph{\i}}$ is injective, it is clear that $\hat H^{n-k+1}(M)$ is finite dimensional whenever
$H^{n-k+1}(M)$ is, and in such case we have
$$
\hat c_{n-k+1}(M)=b_{n-k+1}(M)-b_{n+k+1}(M)+ \dim (\text{im}\,f_{n-k+2}).
$$
Therefore, the properties obtained in Proposition~\ref{propiedades-finitas} extend to the space $\hat H^{n-k+1}(M)$
as follows:

\begin{proposition}\label{propiedades-finitas-para-c-sombrero}
Let $\M$ be a symplectic manifold of finite type and let $1\leq k\leq n$.
The following properties hold for $\hat c_{n-k+1}(M)$:
\begin{enumerate}
\item[{\rm (i)}] \emph{Finiteness and bounds for the coeffective number $\hat c_{n-k+1}(M)$:} the space
$\hat H^{n-k+1}(M)$ is finite dimensional and its dimension satisfies
the inequalities
\begin{equation}\label{des-para-c-sombrero}
b_{n-k+1}(M) - b_{n+k+1}(M)\leq \hat c_{n-k+1}(M)\leq b_{n-k+1}(M).
\end{equation}
\item[{\rm (ii)}] \emph{Symplectic manifolds satisfying the HLC:} if $\M$ satisfies the HLC then
the lower bound in \eqref{des-para-c-sombrero} is attained, i.e.
$$
\hat c_{n-k+1}(M)=b_{n-k+1}(M)-b_{n+k+1}(M).
$$
\item[{\rm (iii)}] \emph{Exact symplectic manifolds:} if $\omega$ is exact then
the upper bound in \eqref{des-para-c-sombrero} is attained, i.e.
$$
\hat c_{n-k+1}(M)=b_{n-k+1}(M).
$$
\item[{\rm (iv)}] \emph{Poincar\'e lemma:} if $U$ is the open unit disk in $\mathbb{R}^{2n}$ with the standard
symplectic form $\omega=\sum_{i=1}^n dx^i\wedge dx^{n+i}$, then
$\hat c_{n-k+1}(U)=0$.
\end{enumerate}
\end{proposition}

%%%\begin{remark}
%%%{\rm
%%%In {\rm (iv)}, if we admit $k=n+1$, then $\hat c_{n-k+1}(U)=\hat c_{0}(U)=b_{0}(U)=1$, that is,
%%%$\hat H^{n-k+1}(U)\cong \mathbb{R}$ for $k=n+1$.
%%%}
%%%\end{remark}

Inspired by the definition of the Euler characteristic of a manifold, we define the following symplectic invariants:

\begin{definition}\label{chi-k}
Let $\M$ be a symplectic manifold of finite type. For each $1\leq k\leq n$, we define
$$\chi^{(k)}(M) = (-1)^{n-k+1} \hat c_{n-k+1}(M)+ \sum^{2n}_{i=n-k+2}(-1)^i \, c_i^{(k)}(M).$$
\end{definition}

The next proposition shows that each $\chi^{(k)}(M)$ is actually a topological invariant of the manifold.
%%%It is worthy to notice that in Section~\ref{sec_ex} we present a nilmanifold admitting two symplectic forms so that their
%%%generalized coeffective cohomologies are not isomorphic.

\begin{proposition}\label{invtopologicos}
Let $(M^{2n},\omega)$ be a symplectic manifold of finite type. For any $1\leq k\leq n$,
$$\chi^{(k)}(M)=\sum_{r=n-k+1}^{n+k}(-1)^{r} \, b_{r}(M).$$
\end{proposition}

\begin{proof}
The long exact sequence \eqref{selarga} implies
$$
0=\hat c_{n-k+1}-b_{n-k+1}+b_{n+k+1}+\sum_{j=2}^{n+k}(-1)^{j-1} (c_{n-k+j}^{(k)}-b_{n-k+j}+b_{n+k+j}).
$$
Writing this sum in terms of $\chi^{(k)}$ we get
$$
(-1)^{n-k+1} \chi^{(k)} + \sum_{j=1}^{n+k}(-1)^j b_{n-k+j} - \sum_{j=1}^{n+k}(-1)^j b_{n+k+j}=0.
$$
Since $b_i=0$ for $i\geq 2n+1$, the previous equality reduces to:
\begin{eqnarray*}
0&=&(-1)^{n-k+1} \chi^{(k)} + \sum_{j=1}^{n+k}(-1)^j b_{n-k+j} - \sum_{j=1}^{n-k}(-1)^j b_{n+k+j}\\
&=& (-1)^{n-k+1} \chi^{(k)} + \sum_{r=n-k+1}^{2n}(-1)^{r-n+k}\, b_{r} - \sum_{r=n+k+1}^{2n}(-1)^{r-n-k}\, b_{r}.
\end{eqnarray*}
Equivalently,
\begin{eqnarray*}
 \chi^{(k)}&=&\sum_{r=n-k+1}^{2n}(-1)^{r} b_{r} - \sum_{r=n+k+1}^{2n}(-1)^{r} b_{r} = \sum_{r=n-k+1}^{n+k}(-1)^{r} b_{r}.
\end{eqnarray*}
\end{proof}

Observe that the Euler characteristic of $M$ is recovered if we allow $k=n+1$ (see Remark~\ref{limit-cases}).

\medskip

\begin{remark}\label{coef-inv-cohom-class}
{\rm
It is clear from the long exact sequence \eqref{selarga} that, for each $k$, the generalized $k$-coeffective cohomology groups
\eqref{todos-los-gen-coef-groups}
are invariants of the de Rham cohomology class $[\omega^k]$ given by the cap product of $[\omega]$ by itself $k$ times.
Even more, if $[\omega^k]\not=0$ and we denote by $[[\omega^k]]$ the
corresponding element in $\mathbb{P}(H^{2k}(M))$, then all the
generalized $k$-coeffective groups
are invariants of $[[\omega^k]]$.
In conclusion, if $\M$ is a symplectic manifold and $\omega^k$ are not exact, then
the generalized coeffective cohomologies only depend on the element $[[\omega]]\in \mathbb{P}(H^{2}(M))$.

}
\end{remark}

From this remark it follows

\begin{lemma}\label{iso-equiv-coef}
Let $F\colon (M,\omega)\longrightarrow (M',\omega')$ be a
diffeomorphism such that $F^*[\omega']=\lambda [\omega]$ for some non-zero $\lambda\in \mathbb{R}$.
Then, for any $1\leq k\leq n$,
$\hat H^{n-k+1}(M')\cong \hat H^{n-k+1}(M)$
and
$H^{q}_{(k)}(M')\cong H^{q}_{(k)}(M)$
for every $q\geq n-k+2$.
\end{lemma}

\medskip

Notice that it suffices to know the de Rham cohomology of $M$ together with the action of $L^k$ on it, in order
to know the generalized $k$-coeffective cohomology. This can be applied in particular to solvmanifolds satisfying the Mostow condition \cite{Mostow},
that is to say, to compact quotients $G/\Gamma$ of solvable
Lie groups $G$ by a lattice $\Gamma$ satisfying that the algebraic
closures $\mathcal{A}(\rm{Ad}_G(G))$ and $\mathcal{A}(\rm{Ad}_G(\Gamma))$ are equal.
In fact, under this condition one has that the natural map $(\bigwedge^* \frg^*,d) \hookrightarrow (\Omega^* (M),d)$
from the Chevalley-Eilenberg complex of the Lie algebra $\frg$ of $G$ to the de Rham complex of the
solvmanifold $M=G/\Gamma$ is a quasi-isomorphism, i.e. $H^q(M)\cong H^q(\frg)$ for any $0\leq q\leq \dim M$.
The following result is straightforward from the long exact sequence in cohomology:

\begin{proposition}\label{calculo-coef-solvariedades}
Let $(M=G/\Gamma,\omega)$ be a $2n$-dimensional symplectic solvmanifold satisfying the Mostow condition.
Let $\frg$ be the Lie algebra of $G$ and let $\omega'\in \bigwedge^2 \frg^*$ be a left-invariant symplectic form
representing the de Rham class $[\lambda \omega]\in H^2(M)$ for some $\lambda\not=0$.
Then, for any $1\leq k\leq n$, the inclusion $\bigwedge^* \frg^* \hookrightarrow \Omega^* (M)$ induces
isomorphisms
$\hat H^{n-k+1}(M,\omega)\cong \hat H^{n-k+1}(\frg,\omega')$
and
$H^{q}_{(k)}(M,\omega)\cong H^{q}_{(k)}(\frg,\omega')$
for every $q\geq n-k+2$.
\end{proposition}

In particular, the previous result holds for nilmanifolds \cite{Nomizu} and in the
completely solvable case \cite{Hattori}, i.e. when the adjoint representation ${\rm ad}_X$ has
only real eigenvalues for all $X \in \frg$.

Note that for the usual coeffective cohomology, i.e. $k=1$ and $q\geq n+1$, this result was proved in \cite{FIL1} (see also \cite{FIL2}).

\medskip

\begin{remark}\label{cohom-mas-en-general}
{\rm
For other results on the de Rham cohomology of compact solvmanifolds $G/\Gamma$, even in the case that the solvable Lie group
$G$ and the lattice $\Gamma$ do not satisfy the Mostow condition, see \cite{CF,Guan}.
Notice that for infra-solvmanifolds Baues proved in \cite{Baues} an analogous result to Nomizu's theorem
about the isomorphism of its cohomology and that of a certain complex of left-invariant forms,
result that is used %%%by Kasuya
in \cite{Kasuya} to study the $1$-coeffective cohomology of certain symplectic aspherical manifolds.
}
\end{remark}

\medskip

\begin{remark}\label{k-truncada}
{\rm
For general symplectic manifolds (not necessarily of finite type) one has for every $q\geq n-k+2$
the following isomorphism
$$
\frac{H^{q}_{(k)}(M)}{H^{q+2k-1}(M)/L^k(H^{q-1}(M))} \cong \ker \{L^k\colon H^{q}(M)\longrightarrow H^{q+2k}(M) \}.
$$
In particular, if the HLC is satisfied then $H^{q}_{(k)}(M) \cong \ker \{L^k\colon H^{q}(M)\longrightarrow H^{q+2k}(M) \}$.
Since any compact K\"ahler manifold satisfies the HLC, we conclude that, for
any $q\geq n-k+2$, the $k$-coeffective group $H^{q}_{(k)}(M)$ is isomorphic to the space of de Rham cohomology classes that
annihilate the class $[\omega^k]$.
For $k=1$ this result was proved by Bouch\'{e} in \cite{B}, where he refers to the latter groups as the \emph{truncated de Rham groups}.
In \cite{FIL1,Kasuya} the relation of the 1-coeffective cohomology
with the truncated de Rham cohomology is also investigated.
%%%
%%% Nota: en el caso ''exacto'' tenemos $\frac{H^{q}_{(k)}(M)}{H^{q+2k-1}(M)} \cong H^{q}(M)$, para $q\geq n-k+2$.
%%%
}
\end{remark}

%%%%%%%%%%%%%%%%%%%%%%%%%%%%%%%%%%%%%%%%%%%%%%%%%%%%%%%%%%%%%%%%%%%%%%%%%%%%%%%%%%%%%%%%%%%%%%%%%%%%%%%%%%%%%%%%%%%
%%%%%%%%%%%%%%%%%%%%%%%%%%%%%%%%%%%%%%%%%%%%%%%%%%%%%%%%%%%%%%%%%%%%%%%%%%%%%%%%%%%%%%%%%%%%%%%%%%%%%%%%%%%%%%%%%%%%
\section{Extension of the generalized coeffective complexes}\label{exten-gen-coef-cohom}
%%%%%%%%%%%%%%%%%%%%%%%%%%%%%%%%%%%%%%%%%%%%%%%%%%%%%%%%%%%%%%%%%%%%%%%%%%%%%%%%%%%%%%%%%%%%%%%%%%%%%%%%%%%%%%%%%%%
%%%%%%%%%%%%%%%%%%%%%%%%%%%%%%%%%%%%%%%%%%%%%%%%%%%%%%%%%%%%%%%%%%%%%%%%%%%%%%%%%%%%%%%%%%%%%%%%%%%%%%%%%%%%%%%%%%%%

\noindent
%%%Recently, Eastwood~\cite{E} has introduced an elliptic extension of the usual coeffective
%%%complex such that the corresponding cohomology groups are isomorphic to the \emph{primitive cohomologies}
%%%defined and studied by Tseng and Yau in~\cite{TY1,TY2}.
In this section, for any $1\leq k\leq n$, we consider an extension of the $k$-coeffective complex,
which is elliptic and
whose cohomology groups will be isomorphic to the \emph{filtered} cohomology groups introduced by Tsai, Tseng and Yau in \cite{TY3}.
For $k=1$ such extension was constructed by Eastwood in~\cite{E}.

Let us fix $k$ such that $1\leq k\leq n$.
For each $q$, let us consider the quotient space
$\check{\Omega}^q_{(k)}(M)=\frac{\Omega^{q}(M)}{L_{\omega}^k(\Omega^{q-2k}(M))}$.
%%%$\check{\Omega}^q_{(k)}(M,\omega)=\frac{\Omega^{q}(M)}{L_{\omega}^k(\Omega^{q-2k}(M))}$.
We denote by
$\check{d}\colon \check{\Omega}^q_{(k)}(M)\longrightarrow \check{\Omega}^{q+1}_{(k)}(M)$
%%%$\check{d}\colon \check{\Omega}^q_{(k)}(M,\omega)\longrightarrow \check{\Omega}^{q+1}_{(k)}(M,\omega)$
the natural map induced by the exterior differential, i.e. $\check{d}(\check{\alpha})=(d\alpha)\!\check{\phantom{i}}=d\alpha+L_{\omega}^k(\Omega^{q-2k+1}(M))$,
for any $\check{\alpha} \in \check{\Omega}^q_{(k)}(M)$.
Then, we have the following complex
\xymatrixcolsep{1.4pc}\xymatrixrowsep{2pc}
\begin{equation}\label{complejo_nuevo_k}
\xymatrix{
0\ar[r]&\ar[r]\Omega^0\ar[r]^-{d}&\Omega^1\ar[r]^-{d}&
\cdots\ar[r]^-{d}&\Omega^{2k-1}\ar[r]^-{\check{d}}&
\check{\Omega}^{2k}_{(k)}\ar[r]^-{\check{d}}&\cdots\ar[r]^-{\check{d}} &
\check{\Omega}^{n+k-2}_{(k)}\ar[r]^-{\check{d}}
&\check{\Omega}^{n+k-1}_{(k)}\ar[d]^{D}\\
0
&\Omega^{2n}\ar[l]&\Omega^{2n-1}\ar[l]_-{\ d}
&\cdots\ar[l]_-{\ d}&\Omega^{2n-2k+1}\ar[l]_-{\ d} &
\fra^{2n-2k}_{(k)}\ar[l]_-{\ d} &
\cdots\ar[l]_-{\ d}
&\fra^{n-k+2}_{(k)}\ar[l]_-{\ d}
&\fra^{n-k+1}_{(k)}\ar[l]_-{\ d}}
\end{equation}
where $D$ is a second-order differential operator defined as $D(\check{\alpha}) = d\gamma$,
being $\gamma$ the unique $(n-k)$-form satisfying $d\alpha=L_{\omega}^k(\gamma)$.
It can be checked that this complex is elliptic in any degree,
however we will not use this fact in what follows
since the main properties of its cohomology groups
will be derived from a long exact sequence as in Section~\ref{gen-coef-cohom}.

Let us denote by $\check{H}^q_{(k)}(M)$ the cohomology groups associated to the complex~\eqref{complejo_nuevo_k} for $0\leq q\leq 2n+2k-1$.
Notice that $\check{H}^q_{(k)}(M)=H^q(M)$ for any $q\leq 2k-2$ and $\check{H}^q_{(k)}(M)=H^{q-2k+1}_{(k)}(M)$ for any $q\geq n+k+1$.

Now, the sequences of complexes
\xymatrixcolsep{1.5pc}\xymatrixrowsep{1.5pc} $$\xymatrix{
&&\Omega^{n+k-1}\ar[r]^-d&\Omega^{n+k}\ar[r]^-d&\Omega^{n+k+1}\ar[r]^-d& \Omega^{n+k+2}\ar[r]^-d& \Omega^{n+k+3}\ar[r]^-d&\cdots\\
&&\Omega^{n-k-1}\ar[r]^-d\ar[u]^-{L_{\omega}^k}&\Omega^{n-k}\ar[r]^-d\ar[u]^-{L_{\omega}^k}&\Omega^{n-k+1}\ar[r]^-d\ar[u]^-{L_{\omega}^k}
& \Omega^{n-k+2}\ar[r]^-d\ar[u]^-{L_{\omega}^k}& \Omega^{n-k+3}\ar[r]^-d\ar[u]^-{L_{\omega}^k}&\cdots\\
\cdots\ar[r]^-{\check{d}}& \check{\Omega}^{n+k-3}_{(k)}\ar[r]^-{\check{d}}& \check{\Omega}^{n+k-2}_{(k)}\ar[r]^-{\check{d}}&\check{\Omega}^{n+k-1}_{(k)}\ar[r]^-D
&\fra^{n-k+1}_{(k)}\ar[u]^-i\ar[r]^-d& \fra^{n-k+2}_{(k)}\ar[u]^-i\ar[r]^-d& \fra^{n-k+3}_{(k)}\ar[u]^-i\ar[r]^-d&\cdots\\
\cdots\ar[r]^-{d}& \Omega^{n+k-3}\ar[u]^-p\ar[r]^-d& \Omega^{n+k-2}\ar[u]^-p\ar[r]^-d&\Omega^{n+k-1}\ar[u]^-p\ar[r]^-d&\Omega^{n+k}\ar[r]^-d&\Omega^{n+k+1}&&\\
\cdots\ar[r]^-{d}& \Omega^{n-k-3}\ar[u]^-{L_{\omega}^k}\ar[r]^-d& \Omega^{n-k-2}\ar[u]^-{L_{\omega}^k}\ar[r]^-d&\Omega^{n-k-1}\ar[u]^-{L_{\omega}^k}\ar[r]^-d
&\Omega^{n-k}\ar[u]^-{L_{\omega}^k}\ar[r]^-d&\Omega^{n-k+1}\ar[u]^-{L_{\omega}^k}&&\\
}$$
where $i$ denotes the inclusion and $p$ the natural projection,
give rise to the following long exact sequence in cohomology:
%%%
%%%\xymatrixcolsep{2.75pc}\xymatrixrowsep{0.8pc} \begin{equation}\label{sexl-E_k}
%%%\xymatrix{
%%%\cdots\ar[r] & H^{n-k-2}(M)\ar[r]^-{L^k} & H^{n+k-2}(M)\ar[r]^-{H(p)} & \check{H}^{n+k-2}_{(k)}(M) &\\
%%%\ar[r]^-{\check{f}_{n-k-1}} & H^{n-k-1}(M)\ar[r]^-{L^k} & H^{n+k-1}(M)\ar[r]^-{H(p)} & \check{H}^{n+k-1}_{(k)}(M) &\\
%%%\ar[r]^-{\check{f}_{n-k}} & H^{n-k}(M)\ar[r]^-{L^k}  & H^{n+k}(M)\ar[r]^-{\check{f}_{n-k+1}} & \check{H}^{n+k}_{(k)}(M) &\\
%%%\ar[r]^-{H(i)} & H^{n-k+1}(M)\ar[r]^-{L^k} & H^{n+k+1}(M)\ar[r]^-{\check{f}_{n-k+2}} & \check{H}^{n+k+1}_{(k)}(M)&\\
%%%\ar[r]^-{H(i)} & H^{n-k+2}(M)\ar[r]^-{L^k} & H^{n+k+2}(M)\ar[r] ^-{\check{f}_{n-k+3}} & \check{H}^{n+k+2}_{(k)}(M)&\!\!\!\!\!\!\cdots.
%%%}
%%%\end{equation}
%%%
\xymatrixcolsep{2.75pc}\xymatrixrowsep{0.8pc}
$$\xymatrix{
\cdots\ar[r]^-{\check{f}_{n-k-2}} & H^{n-k-2}(M)\ar[r]^-{L^k} & H^{n+k-2}(M)\ar[r]^-{H(p)} & \check{H}^{n+k-2}_{(k)}(M) &\\
\ar[r]^-{\check{f}_{n-k-1}} & H^{n-k-1}(M)\ar[r]^-{L^k} & H^{n+k-1}(M)\ar[r]^-{H(p)} & \check{H}^{n+k-1}_{(k)}(M) &}$$
\xymatrixcolsep{3.3pc}\begin{equation}\label{sexl-E_k}\xymatrix{
\ar[r]^-{\check{f}_{n-k}} & H^{n-k}(M)\ar[r]^-{L^k}  & H^{n+k}(M)\ar[r]^-{\check{f}_{n-k+1}} & \check{H}^{n+k}_{(k)}(M) &}
\end{equation}
\xymatrixcolsep{2.75pc}$$\xymatrix{
\ar[r]^-{H(i)} & H^{n-k+1}(M)\ar[r]^-{L^k} & H^{n+k+1}(M)\ar[r]^-{\check{f}_{n-k+2}} & \check{H}^{n+k+1}_{(k)}(M)&\\
\ar[r]^-{H(i)} & H^{n-k+2}(M)\ar[r]^-{L^k} & H^{n+k+2}(M)\ar[r] ^-{\check{f}_{n-k+3}} & \check{H}^{n+k+2}_{(k)}(M)&\!\!\!\!\!\!\!\!\!\!\!\!\!\!\!\!\!\!\!\!\!\!\!\!\!\!\!\!\!\!\cdots.
}
$$
Here $H(i)$ and $H(p)$ are the homomorphisms induced in cohomology by $i$ and $p$, respectively, and
$\check{f}_q$ are the connecting homomorphisms, which are given as follows:

$\bullet$ for any $j\leq n+k-1$ and $[\alpha]\in\check H^j_{(k)}(M)$: $\check f_{j-2k+1}([\alpha]) = [\beta]$, where $d\alpha = L^k_{\omega}(\beta)$;

$\bullet$ for any $j\geq n+k$ and $[\alpha]\in H^j(M)$: $\check f_{j-2k+1}([\alpha]) = [d\beta]$,
where $\alpha = L^k_{\omega}(\beta)$.

\medskip

Let $\check{c}_q^{(k)}(M)$ be the dimension of $\check{H}^q_{(k)}(M)$ when it is finite.
As in Section~\ref{gen-coef-cohom},
using five-term exact sequences from \eqref{sexl-E_k},
we arrive at the following result, that provides an extension of Proposition~\ref{propiedades-finitas}.

\begin{proposition}\label{propiedades-finitas-suc-check}
Let $\M$ be a symplectic manifold of finite type and let $1\leq k\leq n$.
Then, for every $0\leq q\leq 2n+2k-1$, the following properties hold:
\begin{enumerate}
\item[{\rm (i)}] \emph{Finiteness and bounds for the numbers $\check{c}_q^{(k)}(M)$:} the group
$\check{H}^q_{(k)}(M)$ is finite dimensional and its dimension $\check{c}_q^{(k)}(M)$ satisfies
the inequalities
\begin{equation}\label{des-check}
b_{q-2k+1}(M)-b_{q+1}(M)\leq \check{c}_q^{(k)}(M)\leq b_{q-2k+1}(M)+b_{q}(M).
\end{equation}
\item[{\rm (ii)}] \emph{Symplectic manifolds satisfying the HLC:} if $\M$ satisfies the HLC then
the lower bound in \eqref{des-check} is attained for every $q\geq n+k$, i.e.
$$
\check{c}_q^{(k)}(M)=b_{q-2k+1}(M)-b_{q+1}(M), \quad\ q\geq n+k.
$$
\item[{\rm (iii)}] \emph{Exact symplectic manifolds:} if $\omega$ is exact then
the upper bound in \eqref{des-check} is attained, i.e.
$$
\check{c}_q^{(k)}(M)=b_{q-2k+1}(M)+b_{q}(M).
$$
\item[{\rm (iv)}] \emph{Poincar\'e lemma:} if $U$ is the open unit disk in $\mathbb{R}^{2n}$ with the standard
symplectic form $\omega=\sum_{i=1}^n dx^i\wedge dx^{n+i}$, then
$\check{c}_0^{(k)}(U)=1=\check{c}_{2k-1}^{(k)}(U)$ and $\check{c}_{q}^{(k)}(U)=0$ for any other value of $q$.
\end{enumerate}
\end{proposition}

\begin{remark}\label{case ii for Lk injective}
{\rm
By (ii) the lower bound in \eqref{des-check} is attained  for every $q\geq n+k$ for symplectic manifolds satisfying the HLC.
Similarly, it can be proved from \eqref{sexl-E_k} that if $\M$ satisfies that all the maps
$L^k\colon H^{n-k}(M)\longrightarrow H^{n+k}(M)$ are \emph{injective}
%%%for any $k$,
then
$\check{c}_q^{(k)}(M)=b_{q}(M)-b_{q-2k}(M)$ for every $q\leq n+k-1$.
In conclusion, if $L^k\colon H^{n-k}(M)\longrightarrow H^{n+k}(M)$ is an isomorphism
for any $1\leq k\leq n$ (for instance, if $\M$ is a \emph{closed} symplectic manifold satisfying the HLC)
then the following equalities hold:
\begin{eqnarray*}
&& \check{c}_q^{(k)}(M)=b_{q}(M)-b_{q-2k}(M),\quad 0\leq q\leq n+k-1;\\[4pt]
&& \check{c}_q^{(k)}(M)=b_{q-2k+1}(M)-b_{q+1}(M),\quad n+k\leq q\leq 2n+2k-1.
\end{eqnarray*}
}
\end{remark}

\begin{example}\label{R2n}
{\rm
By Proposition~\ref{propiedades-finitas-suc-check}~(iv) we have $\check{c}_{2k-1}^{(k)}(U)=1$.
Next we show the non-zero cohomology class generating $\check{H}^{2k-1}_{(k)}(U)$.
Let $\alpha=\sum_{i=1}^n x^i\wedge dx^{n+i}$. The $(2k-1)$-form $\beta=\alpha\wedge\omega^{k-1}$
is $\check{d}$-closed because $d\alpha=\omega$ and hence $d\beta=\omega^{k} \in L_{\omega}^k(\Omega^{0}(U))$.
Clearly, $\beta$ is not $\check{d}$-exact, because it is not $d$-exact. In conclusion, $[\beta]$ defines a non-zero cohomology class and
$\check{H}^{2k-1}_{(k)}(U)=\langle [\beta] \rangle$.
}
\end{example}

\begin{remark}\label{comparison of indexes}
{\rm
Notice that the generalized coeffective space $\hat H^{n-k+1}(M)$ is isomorphic to a quotient of $\check{H}^{n+k}_{(k)}(M)$;
concretely,
\begin{equation}\label{rel-hat-check}
\hat H^{n-k+1}(M)\cong \frac{\check{H}^{n+k}_{(k)}(M)}{H^{n+k}(M)/L^k(H^{n-k}(M))}.
\end{equation}
}
\end{remark}

\medskip

Let $\M$ be a symplectic manifold of finite type. For every $1\leq k\leq n$, we define
$$\check{\chi}^{(k)}(M) = \sum^{2n+2k-1}_{i=0}(-1)^i \, \check{c}_i^{(k)}(M).$$
Let us write $\check{\chi}^{(k)}(M) = \check{\chi}^{(k)}_+(M) + \check{\chi}^{(k)}_-(M)$,
where
$$\check{\chi}^{(k)}_+(M)=\sum^{n+k-1}_{i=0}(-1)^i \, \check{c}_i^{(k)}(M),\quad
\mbox{ and }\quad
\check{\chi}^{(k)}_-(M)=\sum^{2n+2k-1}_{i=n+k}(-1)^i\, \check{c}_i^{(k)}(M).$$
%%%These symplectic invariants satisfy

\begin{proposition}\label{chis-filtered}
Let $\M$ be of finite type. For every $1\leq k\leq n$:
\begin{enumerate}
\item[{\rm (i)}] $\check{\chi}^{(k)}(M) = 0$; consequently, $\check{\chi}^{(k)}_-(M)=-\check{\chi}^{(k)}_+(M)$.
%%%Notice that in particular, if $M$ is closed then $\check{\chi}^{(k)}(M)=0$, as it was proved in \cite{TY3}.
\item[{\rm (ii)}] $\check{\chi}^{(k)}_+(M)=(-1)^{n+k+1}(\check{c}^{(k)}_{n+k}(M)-\hat{c}_{n-k+1}(M)) + \chi^{(k)}(M)$.
\end{enumerate}
\end{proposition}

\begin{proof}
Property (i) follows from~\eqref{sexl-E_k} arguing similarly to the proof of Proposition~\ref{invtopologicos}.

For the proof of (ii), taking into account that $(-1)^{n+k+s}\check{c}^{(k)}_{n+k+s}(M)=-(-1)^{n-k+s+1}c^{(k)}_{n-k+s+1}(M)$
%%%and that $(-1)^{n+k+s}=-(-1)^{n-k+s+1}$
for $s\geq 1$,
by Proposition~\ref{invtopologicos} we get
$\chi^{(k)}(M)+\check{\chi}^{(k)}_-(M)= (-1)^{n+k}(\check{c}^{(k)}_{n+k}(M)-\hat{c}_{n-k+1}(M))$.
Since $\check{\chi}^{(k)}_+(M)=-\check{\chi}^{(k)}_-(M)$, relation (ii) follows.
\end{proof}

Equality (ii) in the above proposition means that the behaviour
of the symplectic invariant $\check{\chi}^{(k)}_+(M)$ only depends on
$\check{c}^{(k)}_{n+k}(M)-\hat{c}_{n-k+1}(M)$, because $\chi^{(k)}(M)$ is a topological invariant
by Proposition~\ref{invtopologicos}.
Moreover, one has the following characterization of the HLC in terms of $\check{\chi}^{(k)}_+(M)$,
which in particular implies that the HLC is determined by the cohomology of the first half
of the complexes \eqref{complejo_nuevo_k}.

\begin{corollary}\label{HLC-caracterizacion}
A symplectic manifold $\M$ of finite type satisfies the HLC if and only if
$\check{\chi}^{(k)}_+(M)=\chi^{(k)}(M)$ for every $1\leq k\leq n$.
\end{corollary}

\begin{proof}
By \eqref{rel-hat-check}, a symplectic manifold satisfies the HLC if and only if
$\hat H^{n-k+1}(M) \cong \check{H}^{n+k}_{(k)}(M)$ for every $1\leq k\leq n$.
Therefore, if $M$ is of finite type then, $\M$ satisfies the HLC if and only
if $\hat{c}_{n-k+1}(M)=\check{c}^{(k)}_{n+k}(M)$ for every $1\leq k\leq n$.
By Proposition~\ref{chis-filtered}~(ii), this is equivalent to
$\check{\chi}^{(k)}_+(M)=\chi^{(k)}(M)$ for every $1\leq k\leq n$.
\end{proof}

\begin{remark}\label{comparison with filtered}
{\rm
In \cite[Theorem 3.1]{TY3} Tsai, Tseng and Yau have introduced elliptic differential complexes of filtered forms
that extend the complex of primitive forms \cite[Proposition 2.8]{TY2}.
Complexes \eqref{complejo_nuevo_k} can be thought as a coeffective version of such filtered complexes.
Moreover,
in \cite[Theorem 4.2]{TY3} they obtain long exact sequences that provide
a resolution of the Lefschetz maps $L^k$. Comparing with \eqref{sexl-E_k}, which also gives
a resolution of the same Lefschetz maps, one immediately concludes that the cohomology $\check{H}^{*}_{(k)}(M)$
is isomorphic to the $(k-1)$-filtered cohomology as follows:

\medskip

$\bullet$ $\check{H}^{n+k-s}_{(k)}(M) \cong F^{k-1}H^{n+k-s}_+(M)$,\, for $s=1,\ldots,n+k$,

\medskip

$\bullet$ $\check{H}^{n+k+s}_{(k)}(M) \cong F^{k-1}H^{n+k-s-1}_-(M)$,\, for $s=0,1,\ldots,n+k-1$.

\medskip

\noindent In particular, for any $k\geq 1$ one has the following isomorphism between
the $(k-1)$-filtered cohomology group
%%%$F^{k-1}H^{n+k-s-1}_-(M)$
and the $k$-coeffective cohomology group
%%%$H^{n-k+s+1}_{(k)}(M)$
%%%for every $s\geq 1$, that is,
\begin{equation}\label{rel-otra}
F^{k-1}H^{n+k-s-1}_-(M)\cong H^{n-k+s+1}_{(k)}(M) \cong \check{H}^{n+k+s}_{(k)}(M), \quad 1\leq s\leq n+k-1.
\end{equation}

\noindent For $k=1$ we recover the isomorphisms proved in \cite{E} between the extended coeffective cohomology of Eastwood and the primitive cohomology
$PH=F^0H$ of Tseng-Yau \cite{TY1,TY2}. More generally, the isomorphism for any
primitive cohomology group is as follows:
$$
%%%\begin{equation}\label{comparar-con-primitiva}
\begin{array}{l}
P H^q_{\partial_+}(M) \cong F^{0}H^{q}_+(M) \cong \check{H}^{q}_{(1)}(M), \quad\  0\leq q\leq n-1;\\[7pt]
P H^q_{\partial_-}(M) \cong F^{0}H^{q}_-(M) \cong \check{H}^{2n-q+1}_{(1)}(M) \cong H^{2n-q}_{(1)}(M), \quad\ 0\leq q\leq n-1;\\[7pt]
%%%
P H^{n-k+1}_{dd^\Lambda}(M) \cong F^{k-1}H^{n+k-1}_+(M) \cong \check{H}^{n+k-1}_{(k)}(M), \quad\ 1\leq k\leq n;\\[7pt]
P H^{n-k+1}_{d+d^\Lambda}(M) \cong F^{k-1}H^{n+k-1}_-(M) \cong \check{H}^{n+k}_{(k)}(M), \quad\ 1\leq k\leq n.
\end{array}
%%%\end{equation}
$$

From now on, due to the above identifications,
we will refer to the cohomology groups $\check{H}^{q}_{(k)}(M)$ as the filtered cohomology groups of $\M$.
}
\end{remark}

\medskip

Notice that an analogous observation as Remark~\ref{coef-inv-cohom-class} is also valid for the filtered cohomologies.
Hence, a similar result to Lemma~\ref{iso-equiv-coef} holds:

\begin{lemma}\label{iso-equiv-filtered}
Let $F\colon (M,\omega)\longrightarrow (M',\omega')$ be a
diffeomorphism such that $F^*[\omega']=\lambda [\omega]$ for some non-zero $\lambda\in \mathbb{R}$.
Then, for any $1\leq k\leq n$,
$\check{H}^q_{(k)}(M') \cong \check{H}^q_{(k)}(M)$ for every $0\leq q\leq 2n+2k-1$.
\end{lemma}

A similar result to Proposition~\ref{calculo-coef-solvariedades} for computation
of the filtered cohomologies of certain solvmanifolds is also available:

\begin{proposition}\label{calculo-filtered-solvariedades}
Let $(M=G/\Gamma,\omega)$ be a $2n$-dimensional symplectic solvmanifold satisfying the Mostow condition.
Let $\frg$ be the Lie algebra of $G$ and let $\omega'\in \bigwedge^2 \frg^*$ be a left-invariant symplectic form
representing the de Rham class $[\lambda \omega]\in H^2(M)$ for some $\lambda\not=0$.
Then, for any $1\leq k\leq n$, the inclusion $\bigwedge^* \frg^* \hookrightarrow \Omega^* (M)$ induces
isomorphisms
$\check{H}^q_{(k)}(M,\omega) \cong \check{H}^q_{(k)}(\frg,\omega')$ for every $0\leq q\leq 2n+2k-1$.
\end{proposition}

\begin{remark}\label{BCyA}
{\rm
In \cite{TY1} Tseng and Yau introduced and studied more generally Bott-Chern and Aeppli type cohomologies using
$d$ and $d^\Lambda$ for a symplectic manifold. A characterization of the HLC in the compact case from an
\emph{\`a la Fr\"olicher} inequality is given in \cite{Angella-T}.
Note that for the Bott-Chern and Aeppli type symplectic cohomologies,
a similar result to Proposition~\ref{calculo-filtered-solvariedades} is obtained in
\cite[Theorem 3]{Macri} (see also \cite[Theorem 2.31]{Angella-K})
by using another argument.
}
\end{remark}

%%%%%%%%%%%%%%%%%%%%%%%%%%%%%%%%%%%%%%%%%%%%%%%%%%%%%%%%%%%%%%%%%%%%%%%%%%%%%%%%%%%%%%%%%%%%%%%%%%%%%%%%%%%%%%%%%%%
%%%%%%%%%%%%%%%%%%%%%%%%%%%%%%%%%%%%%%%%%%%%%%%%%%%%%%%%%%%%%%%%%%%%%%%%%%%%%%%%%%%%%%%%%%%%%%%%%%%%%%%%%%%%%%%%%%%%
\section{Relations with the symplectically harmonic cohomology}\label{relacion-con-armonica}
%%%%%%%%%%%%%%%%%%%%%%%%%%%%%%%%%%%%%%%%%%%%%%%%%%%%%%%%%%%%%%%%%%%%%%%%%%%%%%%%%%%%%%%%%%%%%%%%%%%%%%%%%%%%%%%%%%%
%%%%%%%%%%%%%%%%%%%%%%%%%%%%%%%%%%%%%%%%%%%%%%%%%%%%%%%%%%%%%%%%%%%%%%%%%%%%%%%%%%%%%%%%%%%%%%%%%%%%%%%%%%%%%%%%%%%%

\noindent In this section we relate the symplectically harmonic cohomology with the cohomologies studied in the previous sections.

Let $(M^{2n},\omega)$ be a symplectic manifold of dimension $2n$.
The \emph{symplectic star} operator
$
*\colon \Omega^q(M) \longrightarrow \Omega^{2n-q}(M)
$
is defined by
%%%\begin{equation}\label{*-poisson}
$$
\alpha \wedge (*\beta )=\Lambda^q(\Pi)(\alpha, \beta)\frac{\omega^n}{n!},
$$
%%%\end{equation}
for every $q$-forms $\alpha$ and $\beta$, where $\Pi$ is the bivector field dual to $\omega$, i.e.
the natural Poisson structure associated to $\omega$.

Let $\delta\colon \Omega^q(M)\longrightarrow \Omega^{q-1}(M)$ be the operator given by
$\delta\alpha=(-1)^{q+1}*d*\alpha$, for every $q$-form $\alpha$.
Brylinski proved that $\delta=[i(\Pi),d]$, where $i(\cdot)$ denotes the interior product.

\begin{definition}\cite{Br}
{\rm A form $\alpha$
%%%on a symplectic manifold $(M,\omega)$
is called \emph{symplectically harmonic} if $d\alpha =0=\delta \alpha$.
}
\end{definition}

We denote by $\Omega^q_\hr(M)$ the linear space of symplectically
harmonic $q$-forms.
Unlike the Hodge theory, there are non-zero exact
symplectically harmonic forms. Now, following Brylinski
\cite{Br}, one defines the \emph{symplectically harmonic cohomology}
$$
H^q_\hr(M)=\frac{\Omega^q_\hr(M)}{\Omega^q_\hr(M)\cap \text{im}\, d},
$$
for $0\leq q\leq 2n$.
Hence, $H^q_\hr(M)$ is the subspace of the $q$-th de Rham cohomology group consisting of all the de Rham cohomology
classes of degree $q$ containing a symplectically harmonic representative.
By analogy with the Hodge theory, Brylinski~\cite{Br} conjectured that any de Rham cohomology class
admits a symplectically harmonic representative.
Mathieu~\cite{Ma} (and independently Yan~\cite{Yan}) proved that Brylinski conjecture holds,
namely $H^{q}_\hr(M,\omega) = H^q(M)$ for every $0\leq q\leq 2n$, if and only if $\M$ satisfies the~HLC.

An important result is that for any symplectic manifold every de Rham cohomology class up to degree~2 admits
a symplectically harmonic representative~\cite{Yan} (see also \cite{IRTU1} for more general results),
that is, $H^q_\hr(M)=H^q(M)$ for $q=0,1,2$.
For every $q \leq n$, if we set
$$
P^q(M)=\{[\alpha]\in H^q(M)\,|\,L^{n-q+1}[\alpha]=0\},
$$
then $P^{q}(M) \subset H^{q}_\hr(M)$ \cite{Yan}.
Moreover, the following result (proved in \cite[Corollary 2.4]{IRTU1} and \cite[Lemma 4.3]{Ym})
gives a description of the spaces $H^q_\hr(M)$:

\begin{theorem}\label{harm}
Let $\M$ be a symplectic manifold of dimension $2n$. Then,
\begin{eqnarray}
&&H^{q}_\hr(M) = P^{q}(M,\omega)+L(H^{q-2}_\hr(M)),
\quad \mbox{ for } 0\leq q \leq n;\nonumber\\[5pt]
&&H^{q}_\hr(M) = {\rm Im}\,\{L^{q-n}\colon H^{2n-q}_\hr(M) \longrightarrow H^{q}(M)\},
\quad \mbox{ for } n+1\leq q \leq 2n.\nonumber
\end{eqnarray}
%%\begin{eqnarray}
%%H^{n-k}_\hr(M,\omega)&=&P^{n-k}(M,\omega)+L(H^{n-k-2}_\hr(M,\omega)),
%%\quad \mbox{ for } 0\leq k \leq n;\nonumber\\[5pt]
%%H^{n+k}_\hr(M,\omega)&=&{\rm Im}\,\{L^k\colon H^{n-k}_\hr(M,\omega) \longrightarrow H^{n+k}(M)\},
%%\quad \mbox{ for } 1\leq k \leq n,\nonumber
%%\end{eqnarray}
%%%%where $$P^q(M,\omega)=\{[\alpha]\in H^q(M)\,|\,L^{n-q+1}[\alpha]=0\}.$$
\end{theorem}

%%\begin{proof}
%%The first equality is proved in \cite[Lemma 4.3]{Ym}, the second
%%equality is proved in \cite[Corollary 2.4]{IRTU1}.
%%\end{proof}

Next we suppose that $\M$ is of finite type and denote by $h_q(M)$ the dimension of $H^q_\hr(M)$.

\begin{example}\label{TM}
{\rm
Let $M^{n}$ be a manifold of dimension $n$ and of finite type, and let $(T^*M,\omega_0)$ be the \emph{cotangent bundle} endowed with the
\emph{standard} symplectic
form. Since $\omega_0$ is exact, the homomorphisms $L^{k}$ are identically zero and by Theorem~\ref{harm} we have
$$h_{q}(T^*M,\omega_0)=b_q(M), \ \mbox{ for } q\leq n,
\quad \mbox{ and }\quad h_{q}(T^*M,\omega_0)=0, \  \mbox{ for } n+1\leq q\leq 2n.$$
For the generalized coeffective cohomology, from Propositions~\ref{propiedades-finitas}~(iii) and~\ref{propiedades-finitas-para-c-sombrero}~(iii) it follows that
$$\hat c_{n-k+1}(T^*M,\omega_0)=b_{n-k+1}(M)$$
and
$$c_q^{(k)}(T^*M,\omega_0)=b_q(M)+b_{q+2k-1}(M), \ \mbox{ for } n-k+2\leq q\leq 2n.$$
Furthermore, from Proposition~\ref{propiedades-finitas-suc-check}~(iii) we get
$$\check{c}_q^{(k)}(T^*M,\omega_0)=b_{q-2k+1}(M)+b_{q}(M), \ \mbox{ for } q\leq 2n+2k-1.$$
}
\end{example}

\medskip

In the following result we relate the generalized coeffective cohomology with the harmonic cohomology
via the
coeffective groups $\hat H^{1}(M),\ldots, \hat H^{n}(M)$.

\begin{theorem}\label{rel_armonica_coef}
Let $\M$ be a symplectic manifold of finite type.
The following relation holds for every $k=1,\ldots,n$:
$$
%%%\begin{equation}\label{relacion_c_h}
h_{n-k+1}(M) - h_{n+k+1}(M) = \hat c_{n-k+1}(M).
%%%\end{equation}
$$
\end{theorem}

\begin{proof}
By Theorem~\ref{harm}, $H^{n-k+1}_{\hr}(M)=P^{n-k+1}(M)+L(H^{n-k-1}_{\hr}(M))$. Hence,
%%%\begin{eqnarray*}
%%%h_{n-k+1}(M) \!&\!\!=\!\!& \dim P^{n-k+1}(M) + \dim L(H^{n-k-1}_{\hr}(M))\\ && - \dim (P^{n-k+1}(M)\cap L(H^{n-k-1}_{\hr}(M))).
%%%\end{eqnarray*}
$$
h_{n-k+1}(M) = \dim P^{n-k+1}(M) + \dim L(H^{n-k-1}_{\hr}(M)) - \dim (P^{n-k+1}(M)\cap L(H^{n-k-1}_{\hr}(M))).
$$

It follows from \eqref{selarga} that $P^{n-k+1}(M)$ is isomorphic to the space $\hat H^{n-k+1}(M)$,
and therefore, $\dim P^{n-k+1}(M)=\hat c_{n-k+1}(M)$. On the other hand,
$$
P^{n-k+1}(M)\cap L(H^{n-k-1}_{\hr}(M)) = \ker L^{k}\big\vert_{L(H^{n-k-1}_{\hr}(M))}.
$$
Now,
\begin{eqnarray*}h_{n-k+1}(M) \!&\!\!=\!\!& \hat c_{n-k+1}(M) + \dim L(H^{n-k-1}_{\hr}(M))
- \dim \left( \ker L^{k}\big\vert_{L(H^{n-k-1}_{\hr}(M))} \right)\\
\!&\!\!=\!\!& \hat c_{n-k+1}(M) +  \dim (L^{k+1}(H^{n-k-1}_{\hr}(M)))\\
\!&\!\!=\!\!& \hat c_{n-k+1}(M) + h_{n+k+1}(M).
\end{eqnarray*}
\end{proof}

From Proposition~\ref{propiedades-finitas-para-c-sombrero} we get directly upper and lower bounds for
the difference $h_{n-k+1}(M) - h_{n+k+1}(M)$. Moreover, the previous theorem, together
with Proposition~\ref{invtopologicos} and \eqref{rel-otra}, provides further relations between the harmonic
and the filtered cohomologies.

Next we derive some concrete relations of the harmonic cohomology with the groups $\check{H}^{q}_{(k)}(M)$.

\begin{proposition}\label{rel_armonica_coef-2}
Let $\M$ be a symplectic manifold of finite type.
For every $1\leq k\leq n$ we have
$$
0\leq \check{c}_{n+k}^{(k)}(M)-\hat c_{n-k+1}(M) \leq b_{n+k}(M) - h_{n+k}(M),
$$
where the latter equality holds if and only if $L^k(H^{n-k}_{\hr}(M))=L^k(H^{n-k}(M))$.
\end{proposition}

\begin{proof}
It follows from \eqref{rel-hat-check} that $\hat c_{n-k+1}(M)= \check{c}_{n+k}^{(k)}(M) - b_{n+k}(M) + \dim L^k(H^{n-k}(M))$.
Taking into account that $\dim L^k(H^{n-k}(M))\geq h_{n+k}(M)$ by Theorem~\ref{harm}, we conclude the relation.
\end{proof}

The following consequence will be useful later for symplectic manifolds of low dimension.

\begin{corollary}\label{rel_armonica_coef-2-cor}
Let $\M$ be a symplectic manifold of finite type.
Then:
\begin{eqnarray}
&&\check{c}_{2n}^{(n)}(M) = b_1(M) + b_{2n}(M) - h_{2n}(M),\nonumber\\[3pt]
&&\check{c}_{2n-1}^{(n-1)}(M) = b_2(M) + b_{2n-1}(M) - h_{2n-1}(M) - h_{2n}(M),\nonumber\\[3pt]
&&\check{c}_{2n-2}^{(n-2)}(M) = b_{2n-2}(M) + h_3(M) - h_{2n-2}(M) -h_{2n-1}(M).\nonumber
\end{eqnarray}
\end{corollary}

\begin{proof}
Notice that $L^k(H^{n-k}_{\hr}(M))=L^k(H^{n-k}(M))$
is satisfied for $k=n, n-1$ and $n-2$ because
$H^q_\hr(M)=H^q(M)$ for $q=0,1,2$. Hence, it suffices to apply Proposition~\ref{rel_armonica_coef-2}
and use Theorem~\ref{rel_armonica_coef}
to relate the coeffective numbers with the harmonic cohomology.
\end{proof}

Next we show some other general properties that we will use later in the following sections.

\begin{proposition}\label{igualdades-general}
Let $\M$ be a symplectic manifold of finite type. Then:
\begin{enumerate}
\item[{\rm (i)}] $\hat c_1(M)=b_1(M)$, and $c_q^{(n)}(M)=b_q(M)$ for every $2\leq q\leq 2n$.
\item[{\rm (ii)}] For any $1\leq k\leq n$, $\check{c}_q^{(k)}(M)=c_{q-2k+1}^{(k)}(M)$ for every $n+k+1\leq q\leq 2n+2k-1$.
\end{enumerate}
\end{proposition}

\begin{proof}
(i) is clear from the definition of the generalized coeffective cohomology for $k=n$
and from the long exact sequence \eqref{selarga}. Equalities (ii) are direct from the definition of  $\check{H}^q_{(k)}(M)$.
\end{proof}

For closed manifolds one has additional relations.
%%%for the symplectic cohomologies.
For instance, the numbers $\check c_q^{(k)}(M)$ satisfy certain duality (see \cite[Proposition 4.8]{TY3}
for the corresponding duality for the filtered cohomology groups), whereas the harmonic number
$h_{2n-1}(M)$ is always even \cite[Lemma 1.14]{IRTU1}. The proof of these facts follows from the existence
of the usual non-singular pairing $p([\alpha],[\beta])=\int_M \alpha\wedge\beta$,
for $[\alpha]\in H^q(M)$ and $[\beta]\in H^{m-q}(M)$, valid on any closed $m$-dimensional manifold $M$.
In the following proposition we collect these results together with other properties.

\begin{proposition}\label{igualdades}
Let $\M$ be a closed symplectic manifold. Then:
\begin{enumerate}
\item[{\rm (i)}] For any $1\leq k\leq n-1$,  $c_q^{(k)}(M)=b_q(M)$ for every $2n-2k+1\leq q\leq 2n$.
\item[{\rm (ii)}] For any $1\leq k\leq n$,
$\check c_q^{(k)}(M)=\check c_{2n+2k-q-1}^{(k)}(M)$ for every $0\leq q\leq n+k-1$.
\item[{\rm (iii)}] $h_{2n}(M)=b_{2n}(M)$, and $h_{2n-1}(M)$ is always even.
\item[{\rm (iv)}] $h_{n-1}(M) - (\check c_{n+1}^{(1)}(M)-\hat c_n(M)) \leq h_{n+1}(M)\leq h_{n-1}(M)$.
\end{enumerate}
\end{proposition}

\begin{proof}
By definition of the coeffective cohomology one always has that
$c_q^{(k)}(M)=b_q(M)$ for any $1\leq k\leq n-1$ and for every $q\geq 2n-2k+2$.
Moreover, since $M$ is closed, $H^{2n}(M)=\langle [\omega^n] \rangle$
and $L^k\colon H^{2n-2k}(M)\longrightarrow H^{2n}(M)$ is surjective, so the long exact sequence
\eqref{selarga} implies $c_{2n-2k+1}^{(k)}(M)=b_{2n-2k+1}(M)$. This proves~(i).

Property (ii) follows from \cite[Proposition 4.8]{TY3} taking into account the identifications
given in Remark~\ref{comparison with filtered}.

The proof of (iii) is a consequence of the fact that the rank of $L^{n-1}\colon H^1(M)\longrightarrow H^{2n-1}(M)$ is
always an even number \cite[Lemma 1.14]{IRTU1}.

To prove (iv),
since $H^{n+1}_{\hr}(M) = L(H^{n-1}_{\hr}(M))$, we have
$$
h_{n-1}(M) = h_{n+1}(M) + \dim  \left(\ker \{L\colon H^{n-1}_{\hr}(M) \to H^{n+1}(M) \}\right).
$$
Therefore,
\begin{eqnarray*}
h_{n-1}(M) - h_{n+1}(M) &=& \dim  \left(\ker \{L\colon H^{n-1}_{\hr}(M) \to H^{n+1}(M) \}\right) \\
&\leq& \dim  \left(\ker \{L\colon H^{n-1}(M) \to H^{n+1}(M) \}\right)\\
&=& b_{n-1}(M) - \dim \left(\text{im}\,\{L\colon H^{n-1}(M)\to H^{n+1}(M)\}\right)\\
&=& b_{n+1}(M) - \dim \left(\text{im}\,\{L\colon H^{n-1}(M)\to H^{n+1}(M)\}\right)\\
&=& \check c_{n+1}^{(1)}(M)-\hat c_n(M),
\end{eqnarray*}
where the last equality follows from Remark~\ref{comparison of indexes}.
\end{proof}

Notice that Theorem~\ref{rel_armonica_coef} does not provide any relation for $h_{n+1}$, so (iv) in the
above proposition provides upper and lower bounds for the harmonic number $h_{n+1}$ on closed symplectic manifolds.

\begin{corollary}\label{n_n-1_coeffective}
Let $\M$ be a closed symplectic manifold. Then,
\begin{equation}\label{coefectivos-topologicos}
\hat c_1(M)=\check{c}_{2n}^{(n)}(M)=b_1(M) \quad\mbox{ and }\quad \hat c_2(M)=b_2(M)-1.
\end{equation}
Hence, the generalized coeffective cohomology groups \eqref{todos-los-gen-coef-groups} for $k=n$ and $n-1$,
as well as the $n$-filtered cohomology groups are topological invariants.
\end{corollary}

\begin{proof}
%Proposition~\ref{igualdades-general}~(i) says that $\hat c_1(M)=b_1(M)$...
The formulas for
$\chi^{(n)}$ and $\chi^{(n-1)}$ given in Proposition~\ref{invtopologicos}
together with part (i) in Proposition~\ref{igualdades-general} and Proposition~\ref{igualdades} imply
\eqref{coefectivos-topologicos}, so all the generalized $n$- and $(n-1)$-coeffective numbers
are topological.
For the $n$-filtered cohomology it suffices to use Corollary~\ref{rel_armonica_coef-2-cor}
and the fact that $\check c^{(n)}_q = c^{(n)}_{q-2n+1}$ for $2n+1\leq q\leq 4n-1$.
\end{proof}

We finish this section noticing that from Theorem~\ref{harm} it follows directly that an analogous result to Lemmas~\ref{iso-equiv-coef}
and~\ref{iso-equiv-filtered} also holds for the symplectically harmonic cohomology \cite[Proposition~1]{Ym},
and in the case of solvmanifolds satisfying the Mostow condition, a result similar to
Propositions~\ref{calculo-coef-solvariedades} and~\ref{calculo-filtered-solvariedades}
is also valid (this was first observed in \cite[Proposition~2]{Ym}, see also \cite{IRTU1,IRTU2}, for
the class of symplectic nilmanifolds).

%%%%%%%%%%%%%%%%%%%%%%%%%%%%%%%%%%%%%%%%%%%%%%%%%%%%%%%%%%%%%%%%%%%%%%%%%%%%%%%%%%%%%%%%%%%%%%%%%%%%%%%%%%%%%%%%%%%
%%%%%%%%%%%%%%%%%%%%%%%%%%%%%%%%%%%%%%%%%%%%%%%%%%%%%%%%%%%%%%%%%%%%%%%%%%%%%%%%%%%%%%%%%%%%%%%%%%%%%%%%%%%%%%%%%%%%
\section{Symplectic flexibility of closed manifolds}\label{flex}
%%%%%%%%%%%%%%%%%%%%%%%%%%%%%%%%%%%%%%%%%%%%%%%%%%%%%%%%%%%%%%%%%%%%%%%%%%%%%%%%%%%%%%%%%%%%%%%%%%%%%%%%%%%%%%%%%%%
%%%%%%%%%%%%%%%%%%%%%%%%%%%%%%%%%%%%%%%%%%%%%%%%%%%%%%%%%%%%%%%%%%%%%%%%%%%%%%%%%%%%%%%%%%%%%%%%%%%%%%%%%%%%%%%%%%%%

\noindent In this section we focus on closed symplectic manifolds, for which we introduce
a notion of flexibility for the generalized coeffective and filtered cohomologies, 
as analogous notions of the concept of harmonic flexibility introduced and studied in \cite{IRTU1, Yan}
and motivated by a question raised by Khesin and McDuff.
Furthermore, we study their relations with the harmonic flexibility.

In what follows, $M$ will refer to a closed smooth manifold admitting symplectic forms. 

\begin{definition}\label{flexible}
{\rm
A $2n$-dimensional $M$ is said to be
\begin{enumerate}
\item[{\rm (i)}] \texttt{c}-\emph{flexible}, if $M$ possesses a continuous family of symplectic forms $\omega_t$, where $t\in [a,b]$,
such that $\hat c_{n-k+1}(M,\omega_a)\not=\hat c_{n-k+1}(M,\omega_b)$ or $c_{q}^{(k)}(M,\omega_a)\not=c_{q}^{(k)}(M,\omega_b)$ for some
$1\leq k\leq n$ and $n-k+2\leq q\leq 2n$;
\item[{\rm (ii)}] \texttt{f}-\emph{flexible}, if $M$ possesses a continuous family of symplectic forms $\omega_t$, $t\in [a,b]$,
such that $\check c_{q}^{(k)}(M,\omega_a)\not=\check c_{q}^{(k)}(M,\omega_b)$ for some $1\leq k\leq n$ and $0\leq q\leq 2n+2k-1$;
\item[{\rm (iii)}] \texttt{h}-\emph{flexible}, if $M$ possesses a continuous family of symplectic forms $\omega_t$, $t\in [a,b]$,
such that $h_{q}(M,\omega_a)\not=h_{q}(M,\omega_b)$ for some $0\leq q\leq 2n$.
\end{enumerate}
}
\end{definition}

Notice that \texttt{h}-flexible manifolds are precisely the flexible manifolds in \cite{IRTU1}.

\medskip

Since $H^q_\hr(M)=H^q(M)$ for $q=1,2$,
%%%by Theorem~\ref{harm}
we have
$h_{2n-q}(M)=\dim ({\rm Im}\,\{L^{n-q}\colon H^{q}(M) \longrightarrow H^{2n-q}(M)\})$
for $q=1,2$ by Theorem~\ref{harm}.
Now, if $\omega_t$ is a continuous family of symplectic structures on $M$, $t\in [a,b]$,
then it is clear that
\begin{equation}\label{lower-semicontinuous}
h_{2n-1}(M,\omega_t) \geq h_{2n-1}(M,\omega_a)\quad \mbox{ and }\quad h_{2n-2}(M,\omega_t) \geq h_{2n-2}(M,\omega_a),
\end{equation}
that is, these symplectically harmonic numbers satisfy a ``lower-semicontinuous'' property.

In the following result we observe that an ``upper-semicontinuous'' property holds for
all the coeffective and the filtered numbers.

\begin{proposition}\label{upper-semicontinuous}
Let $\omega_t$ be a continuous family of symplectic structures on $M$, for $t\in [a,b]$. Then, for any $1\leq k\leq n$
the following inequalities hold:
\begin{eqnarray*}
&& \hat c_{n-k+1}(M,\omega_t) \leq \hat c_{n-k+1}(M,\omega_a),\\[5pt]
&& c_{q}^{(k)}(M,\omega_t) \leq c_{q}^{(k)}(M,\omega_a), \mbox{ for every } n-k+2\leq q\leq 2n,\\[5pt]
&& \check{c}_{n+k}^{(k)}(M,\omega_t)\leq \check{c}_{n+k}^{(k)}(M,\omega_a), \mbox{ for every } 0\leq q\leq 2n+2k-1.
\end{eqnarray*}
\end{proposition}

\begin{proof}
It follows directly from the long exact sequences \eqref{selarga} and \eqref{sexl-E_k}.
\end{proof}

Next we study relations among the three different types of flexibility. We begin in dimension four.

%%%DIMENSION 4

\begin{theorem}\label{dim4}
Let $M$ be a $4$-dimensional closed manifold. Then:
\begin{enumerate}
\item[{\rm (i)}] $M$ is never \emph{\texttt{c}}-flexible;
\item[{\rm (ii)}] $M$ is \emph{\texttt{f}}-flexible if and only if it is \emph{\texttt{h}}-flexible.
\end{enumerate}
\end{theorem}

\begin{proof}
%%%Suppose that $M$ has a symplectic form $\omega$. 
Since $n=2$, we need to study the coeffective numbers $c_q^{(1)}$ for $q=3,4$, $c_q^{(2)}$ for $q=2,3,4$, $\hat c_1$ and $\hat c_2$. 
%we have $k=1$ or $2$.
Proposition~\ref{igualdades-general}~(i), Proposition~\ref{igualdades}~(i) and
Corollary~\ref{n_n-1_coeffective} imply that $c_{2}^{(2)}=b_2$, $c_{3}^{(1)}=c_{3}^{(2)}=b_3$, $c_{4}^{(1)}=c_{4}^{(2)}=b_4$,
$\hat c_{1}=b_1$ and $\hat c_{2}=b_2-1$. Therefore, $M$ cannot be \texttt{c}-flexible and (i) is proved.

By Proposition~\ref{igualdades-general}~(ii) we have $\check{c}_q^{(k)}=c_{q-2k+1}^{(k)}$ for $k=1,2$ and for any $3+k\leq q\leq 3+2k$, so they
are topological invariants.
By the duality given in Proposition~\ref{igualdades}~(ii), it remains to study $\check{c}_3^{(1)}$ and $\check{c}_4^{(2)}$.

Since $h_q=b_q$ for $q=0,1,2$ and $4$,
the first two equalities in Corollary~\ref{rel_armonica_coef-2-cor} imply that $\check c_{4}^{(2)}=b_1$ and
$\check{c}_3^{(1)}=b_2-b_0+b_3-h_3$, i.e.
$$
\check{c}_3^{(1)}=b_1+b_2-h_3-1.
$$
Therefore, $M$ is \emph{\texttt{f}}-flexible
iff $M$ is \emph{\texttt{h}}-flexible, since $\check{c}_3^{(1)}(M,\omega_t)$ varies along a family of symplectic forms $\omega_t$
iff $h_3(M,\omega_t)$ varies.
\end{proof}

The previous proof shows that the fundamental relation between flexibilities on a 4-dimensional manifold~$M$ is
$$
%%%\begin{equation}\label{fund-dim-4}
\check{c}_3^{(1)}(M,\omega_t)=b_1(M)+b_2(M)-h_3(M,\omega_t)-1.
%%%\end{equation}
$$

\begin{corollary}\label{no-flex-4-dim}
Let $M$ be a $4$-dimensional closed manifold.
If the first Betti number $b_1(M)\leq 1$ then $M$ is not \emph{\texttt{f}}-flexible.
\end{corollary}

\begin{proof}
Proposition~\ref{igualdades}~(iii) implies that $h_3$ is an even number for any symplectic form. Since $h_3\leq b_3$
then $h_3=0$ and therefore $M$ cannot be flexible.
\end{proof}

In particular there do not exist simply-connected closed 4-manifolds which are \emph{\texttt{f}}-flexible.

On the other hand,
Yan \cite{Yan} proved that there are no \emph{\texttt{h}}-flexible 4-dimensional nilmanifolds.
Even more, one can see that the same holds in the bigger class of completely solvable solvmanifolds.
For that, by the classification given in \cite[Table 2]{Bock} it remains to check that
a solvmanifold based on the Lie algebra
$de^1=e^{13}$, $de^2=-e^{23}$, $de^3=de^4=0$ is not \emph{\texttt{h}}-flexible.
In fact, any invariant symplectic structure is of the form
$\omega=A\,e^{12}+B\,e^{13}+C\,e^{23}+D\,e^{34}$,
with $AD\not=0$, and the number $h_3$ only depends on the element $[[\omega]]$ in $\mathbb{P}(H^2(M))$
(see Remark~\ref{coef-inv-cohom-class}),
so we can suppose that $A=1$ and $B=C=0$ because $[e^{13}]=[e^{23}]=0$.
Thus, it suffices to study the
family $\omega_t=e^{12}+t\,e^{34}$, with $t\not=0$.
A direct calculation shows
$$
L_{[\omega_t]}(H^1(M))=\langle [\omega_t\wedge e^3], [\omega_t\wedge e^4] \rangle=\langle [e^{123}],[e^{124}] \rangle=H^3(M),
$$
i.e.
$h_3(\omega_t)=b_3$ for any $t\not=0$, and therefore the solvmanifold is not \emph{\texttt{h}}-flexible. Hence:

\begin{proposition}\label{no-flex-4-dim-solv}
Any $4$-dimensional completely solvable solvmanifold is not \emph{\texttt{c}}-flexible, \emph{\texttt{f}}-flexible
or \emph{\texttt{h}}-flexible.
\end{proposition}

However, there exist $4$-dimensional closed manifolds which are \emph{\texttt{h}}-flexible,
as it was proved in \cite[Corollary 4.2]{Yan} and \cite[Proposition 3.2]{IRTU1}.
In fact, if $(M^4,\omega)$ is a closed symplectic manifold satisfying the conditions
\begin{enumerate}
\item[(i)] the homomorphism $L\colon H^1(M)\longrightarrow H^3(M)$ is trivial,
\item[(ii)] the cup product $H^1(M)\otimes H^2(M) \longrightarrow H^3(M)$ is non-trivial,
\end{enumerate}
then $M$ is \emph{\texttt{h}}-flexible.
Since Gompf \cite[Observation 7]{Gompf} proved the existence of 4-manifolds satisfying (i) and (ii),
from Theorem~\ref{dim4} it follows that there exists $4$-dimensional closed manifolds which are \emph{\texttt{f}}-flexible.
Moreover, taking symplectic products, we arrive at the following existence result:

%%%\begin{proposition}\label{flex-4-dim-existen}
%%%There exists $4$-dimensional closed manifolds which are \emph{\texttt{f}}-flexible.
%%%\end{proposition}
%%%
%%%From this result, and taking symplectic products, we arrive at the following existence result:

\begin{theorem}\label{existencia-f-flexible-en-cualquier-dimension}
For each $n\geq 2$, there exist $2n$-dimensional \emph{\texttt{f}}-flexible closed manifolds.
More precisely, there exists a $2n$-dimensional closed manifold $M$ with a continuous family of symplectic forms $\omega_t$
such that the dimensions of the primitive cohomology groups $P H^{2}_{d+d^\Lambda}(M,\omega_t)$ and $P H^{2}_{dd^\Lambda}(M,\omega_t)$
vary with respect to~$t$.
\end{theorem}

\begin{proof}
Notice first that $\check c_{2n-1}^{(n-1)}(M)=\dim P H^{2}_{d+d^\Lambda}(M)=\dim P H^{2}_{dd^\Lambda}(M)$, and by
Corollary~\ref{rel_armonica_coef-2-cor} we have
$\check{c}_{2n-1}^{(n-1)}(M) = b_2(M) + b_{1}(M) - h_{2n-1}(M) -1$.

On the other hand, by \cite[Proposition 5.3]{IRTU1} we have the following formula for $h_{2n-1}$ of a product
$(M=N_1\times N_2,\ \omega=\omega_1+\omega_2)$ of two
symplectic manifolds $(N_1,\omega_1)$ and $(N_2,\omega_2)$ of respective dimensions $n_1$ and $n_2$:
$$
h_{2n-1}(M)=h_{2n_1-1}(N_1)+h_{2n_2-1}(N_2),
$$
where $n=n_1+n_2$.

Now, let $N_1$ be a 4-dimensional closed manifold such that $h_3$ varies along a continuous family of symplectic forms
and let $N_2$ be, for instance, any compact K\"ahler manifold. Then, on the product manifold $M$ there is a continuous family of symplectic forms
such that $h_{2n-1}(M)$, and so $\check{c}_{2n-1}^{(n-1)}(M)$, varies.
\end{proof}

As we noticed above, there do not exist simply-connected closed 4-manifolds which are \emph{\texttt{f}}-flexible.
By using a recent result by Cho \cite{Cho}, next we prove the existence of flexible simply-connected closed manifolds
in every dimension greater than or equal to six.

\begin{theorem}\label{existencia-f-flexible-en-cualquier-dimension-simply-connected}
For each $n\geq 3$, there exist $2n$-dimensional simply-connected closed manifolds which are \emph{\texttt{f}}-flexible.
More precisely, there exists a $2n$-dimensional simply-connected closed manifold $M$ with a continuous family of symplectic forms $\omega_t$
such that the dimensions of the primitive cohomology groups $P H^{3}_{d+d^\Lambda}(M,\omega_t)$ and $P H^{3}_{dd^\Lambda}(M,\omega_t)$
vary with respect to~$t$. Moreover, the manifold $M$ is homotopy equivalent to some K\"ahler manifold.
\end{theorem}

\begin{proof}
Let us first recall that Cho proves in \cite[Theorem 1.3]{Cho} the existence of a compact K\"ahler manifold $(X,\omega)$
with $\dim_{\mathbb{C}}X=3$ such that

\vskip.15cm

\hskip-.2cm (1) $X$ is simply-connected,

\vskip.15cm

\hskip-.2cm (2) the odd Betti numbers vanish, i.e. $b_{2k+1}(X)=0$ for every integer $k \geq 0$,

\vskip.15cm

\hskip-.2cm (3) $X$ admits a symplectic form $\sigma$ such that $(X,\sigma)$ does not satisfy the HLC, and

\vskip.15cm

\hskip-.2cm (4) $\sigma$ is deformation equivalent to the K\"ahler form $\omega$, that is, there
is a path $\{\omega_t\}_{0 \leq t \leq 1}$ of symplectic

\vskip.05cm

\hskip.35cm forms such that $\omega_0=\omega$ and $\omega_1=\sigma$.

\vskip.15cm

These properties imply that $h_4(X,\omega_t)$ varies with $t$, and the manifold $X$ is \emph{\texttt{h}}-flexible.
One can see that $\hat c_3=0$ and $c_4^{(1)}=b_2-1$ for any symplectic form on $X$, so it is not \emph{\texttt{c}}-flexible,
but $X$ is \emph{\texttt{f}}-flexible because $\check{c}_4^{(1)}=b_2-h_4$.
Hence $X$ provides an example in dimension 6.
Next we consider this manifold to prove that there are simply-connected closed manifolds in every higher dimension
which are \emph{\texttt{f}}-flexible.

Let $M$ be a closed symplectic manifold of dimension $2n$.
From Remark~\ref{comparison with filtered} and Proposition~\ref{igualdades}~(ii) for $k=n-2$ and $q=2n-3$, we notice that
$\dim P H^{3}_{d+d^\Lambda}(M)=\check c_{2n-2}^{(n-2)}(M)=\check c_{2n-3}^{(n-2)}(M)=\dim P H^{3}_{dd^\Lambda}(M)$.
On the other hand, by
Corollary~\ref{rel_armonica_coef-2-cor}
\begin{equation}\label{form-1}
\check{c}_{2n-2}^{(n-2)}(M) = b_{2n-2}(M) + h_3(M) - h_{2n-2}(M) -h_{2n-1}(M).
\end{equation}

Suppose that $M$ is a product of two symplectic manifolds $(N_1,\omega_1)$ and $(N_2,\omega_2)$ of respective dimensions $n_1$ and $n_2$.
By \cite[Proposition 5.3]{IRTU1} we have the following formulas for the harmonic numbers $h_{2n-1}$ and $h_{2n-2}$ of the manifold
$(M=N_1\times N_2,\ \omega=\omega_1+\omega_2)$:
\begin{equation}\label{form-2}
\begin{array}{lll}
&&h_{2n-1}(M)=h_{2n_1-1}(N_1)+h_{2n_2-1}(N_2),\\[3pt]
&&h_{2n-2}(M)=h_{2n_1-2}(N_1)+h_{2n_1-1}(N_1) h_{2n_2-1}(N_2)+h_{2n_2-2}(N_2).
\end{array}
\end{equation}
where $n=n_1+n_2$.

Now, let $N_1=X$ be the 6-dimensional simply-connected closed manifold described above,
%%%such that $h_4$ varies along a continuous family of symplectic forms
and
let $N_2=\mathbb{CP}^{n_2}$ endowed with the standard K\"ahler structure
defined by its
natural complex structure and the Fubini-Study metric.
Hence, the manifold $M$ is a simply-connected closed manifold of dimension $2n=2(3+n_2)\geq 8$.
Since all the odd Betti numbers of $N_1$ and $N_2$ vanish,
the manifold $M$ also has all its odd Betti numbers equal to zero. This implies that
$h_3(M)=0$, $h_{2n-1}(M)=0$ and $h_{2n-2}(M)=h_{4}(X)+1$ by \eqref{form-2},
because $h_{2n_2-2}(N_2)=b_{2n_2-2}(\mathbb{CP}^{n_2})=1$.
Moreover, by K\"unneth formula one arrives at $b_{2n-2}(M)=b_4(X)+1$.
Therefore, the equality \eqref{form-1} reduces to
$$
\check{c}_{2n-2}^{(n-2)}(M) = b_{4}(X) - h_{4}(X).
$$
By the properties of $X$ described above, we conclude that
on the product manifold $M$ there is a continuous family of symplectic forms
such that $\check{c}_{2n-2}^{(n-2)}(M)$ varies.
\end{proof}

From the proofs of
Theorems~\ref{existencia-f-flexible-en-cualquier-dimension} and~\ref{existencia-f-flexible-en-cualquier-dimension-simply-connected},
the resulting symplectic manifolds are also \emph{\texttt{h}}-flexible.
%%%The existence of $2n$-dimensional closed manifolds which are \emph{\texttt{h}}-flexible
%%%was first proved in \cite{IRTU1}.
It is unclear if \emph{\texttt{f}}-flexibility is implied by \emph{\texttt{h}}-flexibility
in dimension higher than or equal to 8 (see Proposition~\ref{dim2n} below for the general relation).
In contrast, in six dimensions we have:

%%%DIMENSION 6

\begin{theorem}\label{dim6}
Let $M$ be a $6$-dimensional closed manifold. Then:
\begin{enumerate}
\item[{\rm (i)}] If $M$ is \emph{\texttt{c}}-flexible then $M$ is \emph{\texttt{f}}-flexible and \emph{\texttt{h}}-flexible.
\item[{\rm (ii)}] If $M$ is not \emph{\texttt{c}}-flexible then, $M$ is \emph{\texttt{f}}-flexible if and only if it is \emph{\texttt{h}}-flexible.
\end{enumerate}
\end{theorem}

\begin{proof}
%%%Suppose that $M$ has a symplectic form $\omega$. 
Since $n=3$, the coeffective numbers to be studied are: $c_q^{(1)}$ for $q=4,5,6$, $c_q^{(2)}$ for $q=3,4,5,6$, $c_q^{(3)}$ for $q=2,3,4,5,6$, and $\hat c_1$, $\hat c_2$, $\hat c_3$.
%%%Proposition~\ref{igualdades-general}~(i), Proposition~\ref{igualdades}~(i) and
%%%Corollary~\ref{n_n-1_coeffective} imply that $c_{2}^{(3)}=b_2$, $c_{3}^{(2)}=c_{3}^{(3)}=b_3$, $c_{4}^{(2)}=c_{4}^{(3)}=b_4$,
%%%$c_{q}^{(1)}=c_{q}^{(2)}=c_{q}^{(3)}=b_q$ for $q=5,6$,
%%%$\hat c_{1}=b_1$ and $\hat c_{2}=b_2-1$.
Corollary~\ref{n_n-1_coeffective} implies that $\hat c_1,\hat c_2$ and all $c_{q}^{(k)}$ for $k=2,3$ are topological invariants.
Moreover, by Proposition~\ref{igualdades}~(i) we have $c_{q}^{(1)}=b_q$ for $q=5,6$.
Now, the formula for
$\chi^{(1)}$ given in Proposition~\ref{invtopologicos} implies that
\begin{equation}\label{fund-ec-dim6-1}
\hat c_{3}=c_{4}^{(1)}+b_3-b_2-b_1+1.
\end{equation}
Therefore, $M$ is \texttt{c}-flexible if and only if
$M$ possesses a continuous family of symplectic forms $\omega_t$, $t\in [a,b]$,
such that $\hat c_{3}(M,\omega_a)\not=\hat c_{3}(M,\omega_b)$.

By Proposition~\ref{igualdades-general}~(ii) we have $\check{c}_q^{(k)}=c_{q-2k+1}^{(k)}$ for $k=1,2,3$ and for any $4+k\leq q\leq 5+2k$,
so they
are topological invariants except possibly $\check{c}_5^{(1)}$, which satisfies
\begin{equation}\label{fund-ec-dim6-2}
\check{c}_5^{(1)}=c_{4}^{(1)}.
\end{equation}
By the duality given in Proposition~\ref{igualdades}~(ii), it remains to study $\check{c}_4^{(1)}$ and $\check{c}_5^{(2)}$.

Since $h_q=b_q$ for $q=0,1,2,6$,
the equalities in Corollary~\ref{rel_armonica_coef-2-cor} together with
Theorem~\ref{rel_armonica_coef} imply the following relations
\begin{equation}\label{fund-ec-dim6-3-4-5}
\hat c_{3}=h_3-h_5,\quad\ \
\hat c_{3}=\check{c}_4^{(1)}+h_4-b_2,\quad\ \
\check{c}_5^{(2)}=-h_5+b_2+b_1-1.
\end{equation}
%%%\begin{equation}\label{fund-ec-dim6-3-4-5}
%%%\begin{array}{rl}
%%%&\hat c_{3}=h_3-h_5,\\
%%%&\hat c_{3}=\check{c}_4^{(1)}+h_4-b_2,\\
%%%&\check{c}_5^{(2)}=-h_5+b_2+b_1-1.
%%%\end{array}
%%%\end{equation}

Therefore, the fundamental equalities that relate the different cohomologies for closed
6-dimensional manifolds are \eqref{fund-ec-dim6-1}--\eqref{fund-ec-dim6-3-4-5}.
Now, using these relations, a direct argument shows (i) and (ii).
\end{proof}

\begin{corollary}\label{dim6-cor}
A $6$-dimensional closed manifold is \emph{\texttt{f}}-flexible if and only if it is \emph{\texttt{h}}-flexible.
\end{corollary}

\begin{remark}\label{no-reciproco}
{\rm
Notice that there exist closed 6-dimensional manifolds which are \emph{\texttt{f}}-flexible and \emph{\texttt{h}}-flexible,
but not \emph{\texttt{c}}-flexible; that is to say, the converse to (i) in Theorem~\ref{dim6}
does not hold in general. Explicit examples of nilmanifolds satisfying this are
given in Section~\ref{sec_ex}.
}
\end{remark}

In higher dimension we have the following result:

%%%DIMENSION 2n

\begin{proposition}\label{dim2n}
Let $M$ be a closed manifold of dimension $2n\geq 8$.
If $M$ is \emph{\texttt{c}}-flexible then $M$ is \emph{\texttt{f}}-flexible or \emph{\texttt{h}}-flexible.
\end{proposition}

\begin{proof}
If $M$ possesses a continuous family of symplectic forms $\omega_t$, $t\in [a,b]$,
such that $c_{q}^{(k)}(M,\omega_a)\not=c_{q}^{(k)}(M,\omega_b)$ for some
$1\leq k\leq n$ and $n-k+2\leq q\leq 2n$, then it is clear that $M$ is \emph{\texttt{f}}-flexible
by Proposition~\ref{igualdades-general}~(ii).

If $M$ possesses a continuous family of symplectic forms $\omega_t$, $t\in [a,b]$,
such that
$\hat c_{n-k+1}(M,\omega_a)\not=\hat c_{n-k+1}(M,\omega_b)$ for some
$1\leq k\leq n$, then Theorem~\ref{rel_armonica_coef} implies
that
$h_{n-k+1}(M,\omega_a)\not=h_{n-k+1}(M,\omega_b)$
or $h_{n+k+1}(M,\omega_a)\not=h_{n+k+1}(M,\omega_b)$,
therefore $M$ is \emph{\texttt{h}}-flexible.
\end{proof}

%%%%%%%%%%%%%%%%%%%%%%%%%%%%%%%%%%%%%%%%%%%%%%%%%%%%%%%%%%%%%%%%%%%%%%%%%%%%%%%%%%%%%%%%%%%%%%%%%%%%%%%%%%%%%%%%%%%
%%%%%%%%%%%%%%%%%%%%%%%%%%%%%%%%%%%%%%%%%%%%%%%%%%%%%%%%%%%%%%%%%%%%%%%%%%%%%%%%%%%%%%%%%%%%%%%%%%%%%%%%%%%%%%%%%%%
\section{Symplectic 6-dimensional nilmanifolds}\label{sec_ex}
%%%%%%%%%%%%%%%%%%%%%%%%%%%%%%%%%%%%%%%%%%%%%%%%%%%%%%%%%%%%%%%%%%%%%%%%%%%%%%%%%%%%%%%%%%%%%%%%%%%%%%%%%%%%%%%%%%%
%%%%%%%%%%%%%%%%%%%%%%%%%%%%%%%%%%%%%%%%%%%%%%%%%%%%%%%%%%%%%%%%%%%%%%%%%%%%%%%%%%%%%%%%%%%%%%%%%%%%%%%%%%%%%%%%%%%

\noindent In this section we present a complete study of the dimensions of the harmonic, coeffective and filtered cohomology groups
of 6-dimensional symplectic nilmanifolds, see Table~1 below.

The symplectically harmonic numbers $h_4$ and $h_5$ were first computed in~\cite{IRTU1},
whereas $h_3$ was obtained in~\cite{IRTU2} (see Remark~\ref{mistakes} for corrections).
As a consequence of our study, we describe
all \emph{\texttt{c}}-flexible, \emph{\texttt{h}}-flexible or \emph{\texttt{f}}-flexible
6-dimensional nilmanifolds.

In the proof of Theorem~\ref{dim6} we found that for 6-dimensional closed symplectic
manifolds the fundamental
equalities that relate the different cohomologies
are \eqref{fund-ec-dim6-1}--\eqref{fund-ec-dim6-3-4-5}.
In addition, since the Euler
characteristic of a nilmanifold vanishes, we have that
$b_3=2(b_2-b_1+1)$. Therefore, relations \eqref{fund-ec-dim6-1}--\eqref{fund-ec-dim6-3-4-5}
for 6-dimensional symplectic nilmanifolds are
\begin{equation}\label{fund-ec-dim6-nilv}
\hat c_{3}=c_{4}^{(1)}\!+\!b_2\!-\!3 b_1\!+\!3,\ \ \
h_3=\hat c_{3}\!+\!h_5,\ \ \
\check{c}_4^{(1)}=\hat c_{3}\!-\!h_4+b_2,\ \ \
\check{c}_5^{(1)}=c_{4}^{(1)},\ \ \
\check{c}_5^{(2)}=-h_5\!+\!b_2\!+\!b_1\!-\!1.
\end{equation}
Recall that $c_{4}^{(1)}$, $\check{c}_4^{(1)}$ and $\check{c}_5^{(2)}$ are by Remark~\ref{comparison with filtered} dimensions
of primitive cohomology groups; concretely,
$c_{4}^{(1)}=\dim PH^2_{\partial_-}(=\dim PH^2_{\partial_+})$,
$\check{c}_4^{(1)}=\dim PH^3_{d+d^\Lambda}(=\dim PH^3_{d d^\Lambda})$ and
$\check{c}_5^{(2)}=\dim PH^2_{d+d^\Lambda}(=\dim PH^2_{d d^\Lambda})$.

%%%The table from \cite{IRTU1} used the classification of nilpotent Lie algebras given by Salamon~\cite{Sa}.
%%%The last one, in turn, is based on the Morozov classification of $6$-dimensional nilpotent Lie algebras \cite{OV}.

It follows from Propositions~\ref{calculo-coef-solvariedades} and~\ref{calculo-filtered-solvariedades}
that the calculation of all the cohomology groups reduces to the Lie algebra level.
In Table~1 nilmanifolds of dimension 6 admitting symplectic structure appear lexicographically with respect to
the triple $(b_1,b_2,6-s)$, where $b_1$ and $b_2$ are the Betti numbers
(first two columns in the table) and $s$ is the step length (third column).
The fourth column contains the description of the structure of the
nilmanifold;
for instance, the notation $(0,0,12,13,14,15)$ means that there exists a basis
$\{e^i\}_{i=1}^6$ of (invariant) 1-forms such that
$$
d e^1=d e^2=0,\quad d e^3=e^1 \wedge e^2,\quad
d e^4=e^1 \wedge e^3,\quad
d e^5=e^1 \wedge e^4,\quad d e^6=e^1 \wedge e^5.
$$
%
%%%%%%
%%%%%%The column headed $\frh_i$ indicates the notation of the algebra following~\cite{CFGU}.
%
%\par The column headed $\oplus$ indicates the dimensions of the
%irreducible subalgebras in case $\frg$ itself is not irreducible.
\par The next columns show the dimensions of the non-trivial harmonic, coeffective and filtered cohomology groups, that is,
$h_k$ $(k=3,4,5)$, $\hat c_3$, $c_4^{(1)}(=\check{c}_5^{(1)})$, $\check c_4^{(1)}$ and $\check c_5^{(2)}$.
Moreover, the columns contain all the
possible values when $\omega$ runs
over the space $\SSS$ of all invariant symplectic structures on the nilmanifold.

The last column shows the dimension of the space $\SSS$, however the cohomology groups only depend on the cohomology class
of the symplectic form. This fact allows to reduce calculations to a smaller number of parameters,
and furthermore
by Remark~\ref{coef-inv-cohom-class} we can always normalize one of the non-zero coefficients
that parametrize the classes of the symplectic forms.

When there are variations in the dimensions of the cohomology groups, they appear in the table written accordingly
to the lower-semicontinuous property \eqref{lower-semicontinuous} of the harmonic numbers $h_4$ and $h_5$, or to the
upper-semicontinuous property of $\hat c_3$, $c_4^{(1)}$, $\check c_4^{(1)}$ and $\check c_5^{(2)}$
(see Proposition~\ref{upper-semicontinuous}). Notice that the harmonic number $h_3$ does not satisfy any
lower or upper semicontinuous property.
Moreover, the variations in the dimensions of the cohomology groups are in correspondence (in the sense that
we explain in Example~\ref{h24} below),
except for the nilmanifolds $(0,0,0,12,13,23)$ and $(0,0,0,12,13,14+23)$
(see Example~\ref{ex_h11} for details on the latter).

%%% \frh_{24}=(0,0,0,12,14,15+23+24)
%%% \frh_{7}=(0,0,0,12,13,23)
%%% \frh_{11}=(0,0,0,12,13,14+23)

%%% ...cuales son reducibles
%%%
%%%Most of the algebras appearing in the table are irreducible.  However, some of them have a splitting structure; concretely, %%%$\frh_6,\,\frh_9,\,\frh_{21}$ and $\frh_{22}$ are of type $1\oplus 5$, i.e., they are direct sum of a subalgebra of dimension 1 and a subalgebra of %%%dimension 5, $\frh_{17}$ is of type $1\oplus 1\oplus 4$, $\frh_2$ is of type $3\oplus 3$, $\frh_8$ is of type $1\oplus 1\oplus 1\oplus 3$
%%%and finally $\frh_1$ is the abelian Lie algebra.

%We have use the relations explained in the sections above in order to reduce computation.  Concretely:
%\begin{eqnarray*}
%c^{(1)}_4&=&\hat c_3+3b_1-b_2-3,\\
%r_4^{(1)}&=&\hat c_3+\widetilde b_2,\\
%h_3&=&h_5+ \hat c_3,\\
%h_4&=&h_2+\hat c_3-r_4^{(1)}.
%\end{eqnarray*}
%
%Moreover, $\hat c_2=b_2-1$ and therefore, in dimension $6$,
%there are not manifolds that are $2$-coeffectively flexible.

\vskip.3cm

\begin{center} %{\large\bf Six-dimensional real nilpotent Lie algebras admitting symplectic structure}\vspace{5pt}
\renewcommand{\arraystretch}{1.25}
{\footnotesize
\def\no{\hb{\bf--}}
%\begin{table}[h!]\label{tabla}
{%%%\extrarowheight2pt
\begin{tabular}{|c|c|c|c|c|c|c|c|c|c|c|c|}
\hline $b_1$&$b_2$&$6\!-\!s$&Structure& $h_3$&
$h_4$&$h_5$ & $\hat c_3$&$c^{(1)}_4$&$\check c^{(1)}_4$&$\check c^{(2)}_5$&$\dim \SSS$\\
\hline
%2&2&1&(0,0,12,13,14+23,34+52)& &\no  &\no &\no &\no& \no\\  NO TIENE ESTRUCTURA SIMPLECTICA
%%%\extrarowheight10pt
%2&2&1&(0,0,12,13,14,34+52)&  &\no &\no &\no &\no&\no\\   NO TIENE ESTRUCTURA SIMPLECTICA
2&3&1&(0,0,12,13,14,15)&3&2&0& 3&3&4 &4 & 7\\
2&3&1&(0,0,12,13,14+23,24+15)&4,3 &2&0&4,3&4,3&5,4&4 & 7\\
2&3&1&(0,0,12,13,14,23+15)&3&2&0& 3&3&4& 4 &7\\
2&4&2&(0,0,12,13,23,14)&4 &4&0& 4&3&4&5 &8\\
2&4&2&(0,0,12,13,23,14-25)&4 &2,3,4&0& 4&3&6,5,4& 5 &8\\
2&4&2&(0,0,12,13,23,14+25)&4 &4&0& 4&3&4& 5 &8\\
\hline
3&4&2&(0,0,0,12,14-23,15+34)& 2 &2&0& 2&4&4& 6 &7\\
3&5&2&(0,0,0,12,14,15+23)&5&4&2&3&4& 4& 5 &8\\
3&5&2&(0,0,0,12,14,15+23+24)&4,5 &3,4&0,2& 4,3&5,4&6,4& 7,5 &8\\
3&5&2&(0,0,0,12,14,15+24)&5&4 &2 & 3&4&4&5 & 8\\
3&5&2&(0,0,0,12,14,15)&5&4&2 &3&4& 4& 5 &8\\
%3&5&3&(0,0,0,12,13,14+35)& &\no &\no&\no& \no&\no\\   NO TIENE ESTRUCTURA SIMPLECTICA
%3&5&3&(0,0,0,12,23,14+35)& &\no &\no&\no& \no& \no\\   NO TIENE ESTRUCTURA SIMPLECTICA
%3&5&3&(0,0,0,12,23,14-35)& &\no&\no &\no & \no&\no\\   NO TIENE ESTRUCTURA SIMPLECTICA
%3&5&3&(0,0,0,12,14,24)&1+5&\no&\no &\no & \no&\no\\   NO TIENE ESTRUCTURA SIMPLECTICA
3&5&3&(0,0,0,12,13+42,14+23)&5 &3 &0 & 5&6&7& 7 &8\\
3&5&3&(0,0,0,12,14,13+42)&5 &3 &0 &5&6&7& 7 &8\\
3&5&3&(0,0,0,12,13+14,24)&5 &2,3 &0 &5& 6&8,7&7 &8\\
3&6&3&(0,0,0,12,13,14+23)&7,6,5 &3,4 &0 & 7,6,5&7,6,5&9,8,7& 8 &9\\
3&6&3&(0,0,0,12,13,24)&6,5 &5 &0 & 6,5&6,5&7,6&8 & 9\\
3&6&3&(0,0,0,12,13,14)&6,5 &4&0 &6,5&6,5&8,7&8 & 9\\
3&8&4&(0,0,0,12,13,23)&10,9&7,8&0 &10,9&8,7&10&10 & 9\\
\hline
%4&6&3&(0,0,0,0,12,15+34)& &\no &\no&\no& \no&\no\\   NO TIENE ESTRUCTURA SIMPLECTICA
4&7&3&(0,0,0,0,12,15)&6&3&2& 4&6&8&8 & 9\\
4&7&3&(0,0,0,0,12,14+25)&7,6&4&2&5,4&7,6&8,7& 8 &9\\
4&8&4&(0,0,0,0,13+42,14+23)& 8&7&2& 6&7&7& 9 &10\\
4&8&4&(0,0,0,0,12,14+23)& 8&6 &2 &6& 7&8& 9 &10\\
4&8&4&(0,0,0,0,12,34)& 8&7 &2  & 6&7&7& 9 &10\\
4&9&4&(0,0,0,0,12,13)& 10&7,8 &2 & 8&8&10,9& 10 &11\\
\hline
%5&9&4&(0,0,0,0,0,12+34)&1+5&\no&\no &\no & \no& \no\\   NO TIENE ESTRUCTURA SIMPLECTICA
5&11&4&(0,0,0,0,0,12)&13 & 9 &4 & 9&10&11& 11 &12\\
\hline
6&15&5&(0,0,0,0,0,0)& 20 & 15 &6& 14&14&14& 14 &15\\
\hline
\end{tabular}}
%\medskip
%\caption{Symplectic invariants of six-dimensional nilmanifolds}
%%%%\caption{Six-dimensional real nilpotent Lie algebras admitting symplectic structure}
%\end{table}
}

\vskip.15cm

{\bf Table 1.} Symplectic invariants of six-dimensional nilmanifolds

\end{center}

The following result is a direct consequence of Table 1.

\begin{theorem}\label{flexdim6clasif}
\
\begin{enumerate}
\item[{\rm (i)}] There exist seven nilmanifolds of dimension~$6$ that are \emph{\texttt{c}}-flexible (and
therefore \emph{\texttt{f}}-flexible and \emph{\texttt{h}}-flexible).
\item[{\rm (ii)}] There exist three nilmanifolds of dimension~$6$ that are \emph{\texttt{f}}-flexible and \emph{\texttt{h}}-flexible,
but not \emph{\texttt{c}}-flexible.
\end{enumerate}
In conclusion, there exist ten $6$-dimensional nilmanifolds that are \emph{\texttt{f}}-flexible and \emph{\texttt{h}}-flexible.
\end{theorem}

Notice that from (ii) it follows that the converse of Theorem~\ref{dim6} does not hold, that is, \emph{\texttt{c}}-flexibility
is the strongest condition.

\begin{remark}\label{mistakes}
{\rm
In \cite{IRTU1,IRTU2} the following symplectically harmonic numbers need correction:
%%%are wrongly calculated:

$\bullet$ for $(0,0,12,13,14,15)$ the number $h_4$ is equal to 2 (not 3);

$\bullet$ for $(0,0,12,13,14,23+15)$ the number $h_3$ is equal to 3 (not 2);

$\bullet$ for $(0,0,0,12,14,15+23)$ the number $h_3$ is equal to 5 (not 4);

$\bullet$ for $(0,0,0,0,12,14+25)$ the number $h_4$ is equal to 4 (not 3).
}
\end{remark}

\begin{example}\label{h24}
{\rm
Let us consider the $6$-dimensional nilmanifold $(0,0,0,12,14,15+23+24)$.
According to Table~1, this manifold is \emph{\texttt{c}}-flexible, \emph{\texttt{f}}-flexible and \emph{\texttt{h}}-flexible.  In fact, consider the following
continuous family of symplectic structures
$$[\omega_t]=(1-\cos t)[e^{13}] - \cos t\,[e^{16}+e^{25}-e^{34}]+(1-\cos t)[e^{26}-e^{45}],\quad t\in \mathbb{R}.$$
This family was constructed first in~\cite{IRTU1,IRTU2} to show \emph{\texttt{h}}-flexibility, and the symplectic structures
$\omega_{t=0}$ and $\omega_{t=\frac\pi 2}$ were considered in~\cite{TY2} concerning the dimension of the
primitive group $PH^2_{\partial_-}$, i.e.
%%%concerning
$c_4^{(1)}$.

This 6-dimensional nilmanifold is the only one where all the non-trivial coeffective, harmonic and primitive
numbers vary. In Table 1 the variations are in correspondence as follows:

\smallskip

\noindent $\bullet$ $h_3(\omega_{2\pi k})=4$,\, $h_4(\omega_{2\pi k})=3$,\, $h_5(\omega_{2\pi k})=0$,\, $\hat c_3(\omega_{2\pi k})=4$,\,
$c_4^{(1)}(\omega_{2\pi k})=\check c_5^{(1)}(\omega_{2\pi k})=5$,\, $\check c_4^{(1)}(\omega_{2\pi k})=6$,

\smallskip

\hskip-.1cm  $\check c_5^{(2)}(\omega_{2\pi k})=7$,\ for any integer $k$;

\smallskip

\noindent $\bullet$ $h_3(\omega_{t})=5$,\, $h_4(\omega_{t})=4$,\, $h_5(\omega_{t})=2$,\, $\hat c_3(\omega_{t})=3$,\,
$c_4^{(1)}(\omega_{t})=\check c_5^{(1)}(\omega_{t})=4$,\, $\check c_4^{(1)}(\omega_{t})=4$,\, $\check c_5^{(2)}(\omega_{t})=5$,

\smallskip

\hskip-.1cm  for $t\not= 2\pi k$.
}
\end{example}

\begin{example}\label{ex_h11}
{\rm
Let us consider the $6$-dimensional nilmanifold $(0,0,0,12,13,14+23)$.
The de Rham class of any symplectic form is given by
$$[\omega]=A\,[e^{14}] + B\,[e^{15}] +C\,[e^{24}] +D\,[e^{35}] +E\,[e^{16}+e^{25}] +F\,[e^{16}-e^{34}],$$
where $(E+F)(CD+EF)\neq 0$.
Direct computations show that $\hat c_3$ and $h_4$ vary as follows:
$$\hat c_3=\begin{cases}7,\,\, \mbox{ if } D=E+2F=0,\\ 6,\,\,\mbox{ if } D=0,\, E+2F\neq 0,\\ 5,\,\,\mbox{ if } D\neq 0, \end{cases}
\quad\quad
h_4=\begin{cases}3,\,\,\mbox{ if }(E+F)^2=CD+EF,\\ 4,\,\,\mbox{ if }(E+F)^2\neq CD+EF. \end{cases}$$
This nilmanifold satisfies $b_2=3b_1-3$, so \eqref{fund-ec-dim6-nilv} implies that $\check c_5^{(1)}=c_4^{(1)}=\hat c_3$.
Moreover, $h_5=0$ from which we get that $\check c_5^{(2)}=8$ and $h_3=\hat c_3$. Hence,
using that $\check c_4^{(1)} = b_2-h_4+ \hat c_3$ by \eqref{fund-ec-dim6-nilv}, we arrive at
$$\check c_4^{(1)}=\begin{cases}
9,\,\, \mbox{ if } D=E+2F=0,\\
8,\,\, \mbox{ if } D=0,\, E+2F\neq 0,\ \mbox{ or }\ D\neq 0, \,(E+F)^2=CD+EF,\\
7,\,\, \mbox{ if } D\neq 0,\, (E+F)^2\neq CD+EF. \end{cases}$$
%%%
%%%$$\check c_4^{(1)}=\begin{cases}9,\quad \mbox{ if } D=0,\, E+2F=0,\\ 8,\quad \mbox{ if } \begin{cases}D=0,\, E+2F\neq 0,\\
%%%D\neq 0, \,(E+F)^2=CD+EF,\end{cases} \\ 7,\, \mbox{ if } D\neq 0,\, (E+F)^2\neq CD+EF. \end{cases}$$
%%%

As a consequence, concrete families can be constructed. Let us consider the
two-parametric family
$$[\omega_{t,s}]=t\,[e^{35}] +(s+2)\,[e^{16}+e^{25}] -\,[e^{16}-e^{34}],$$
where $t,s\geq 0$. Then, the variations of the dimensions are:
\smallskip

\noindent $\bullet$ $h_3(\omega_{0,0})=\hat c_3(\omega_{0,0})=c_4^{(1)}(\omega_{0,0})=\check c_5^{(1)}(\omega_{0,0})=7$ and
$\check c_4^{(1)}(\omega_{0,0})=9$;

\smallskip

\noindent $\bullet$ $h_3(\omega_{t,0})=\hat c_3(\omega_{t,0})=c_4^{(1)}(\omega_{t,0})=\check c_5^{(1)}(\omega_{t,0})=5$ and
$\check c_4^{(1)}(\omega_{t,0})=7$, for $t>0$;

\smallskip

\noindent $\bullet$ $h_3(\omega_{0,s})=\hat c_3(\omega_{0,s})=c_4^{(1)}(\omega_{0,s})=\check c_5^{(1)}(\omega_{0,s})=6$ and
$\check c_4^{(1)}(\omega_{0,s})=8$, for $s>0$.

\medskip

On the other hand, if we consider the family
$$[\omega_{t}]=[e^{24}] +[e^{35}] +t\,[e^{16}+e^{25}] +[e^{16}-e^{34}],$$
where $t\geq 0$, then the variations are:
\smallskip

\noindent $\bullet$ $h_4(\omega_{0})=3$ and $\check c_4^{(1)}(\omega_{0})=8$;

\smallskip

\noindent $\bullet$ $h_4(\omega_{t})=4$ and $\check c_4^{(1)}(\omega_{t})=7$, for $t>0$.
}
\end{example}

\medskip

For compact K\"ahler manifolds $\M$ and for every $q\geq n+1$,
the coeffective cohomology group $H^q_{(1)}(M)$ is isomorphic to the \emph{$[\omega]$-truncated $q$-th de Rham group}
$\widetilde{H}^q_{[\omega]}(M)=\{[\alpha]\in H^q(M) \mid [\alpha]\cup[\omega]=0\}$,
although Fern\'andez, Ib\'a\~nez and de Le\'on showed that this is no longer true for arbitrary compact symplectic manifolds~\cite{FIL1}.
Kasuya has studied in \cite{Kasuya} certain symplectic aspherical (non-K\"ahler) manifolds for which
\begin{equation}\label{isotrun}
H^q_{(1)}(M)\cong \widetilde{H}^q_{[\omega]}(M) \quad \mbox{ for every } q\geq n+1.
\end{equation}
For 6-dimensional symplectic nilmanifolds one has the following result, which suggests that such isomorphism
might be closely related to a low step of nilpotency:

\begin{proposition}\label{isotruncada}
Let $M$ be a symplectic $s$-step nilmanifold of dimension~$6$.
\begin{enumerate}
\item[{\rm (i)}] If $s\leq 2$, then there exists a symplectic form on $M$ satisfying $\eqref{isotrun}$.
\item[{\rm (ii)}] If $s=5$ then $\eqref{isotrun}$ is never satisfied.
\end{enumerate}
\end{proposition}

\begin{proof}
Since $n=3$ we only need to consider $q=4$. Thus, $\eqref{isotrun}$ holds if and only
if $c_4^{(1)}(M)=\dim \widetilde{H}^4_{[\omega]}(M)=b_2(M)-1$, the latter equality coming from the fact that
$L(H^4(M))=H^6(M)$. Now, the result is a direct consequence of Table 1.
\end{proof}

Notice that there are several 3-step and several 4-step symplectic nilmanifolds of dimension $6$
satisfying~\eqref{isotrun}. In fact,
all the cases in Table 1 where $c_4^{(1)}=b_2-1$ have such property.

\begin{remark}
{\rm
Concerning other (non primitive) cohomology groups, we recall that
in \cite[Table 3]{Angella-K} the dimensions of the symplectic Bott-Chern and Aeppli
cohomologies for a particular choice of symplectic structure on each 6-dimensional nilmanifold
have been computed.
}
\end{remark}

\section{A symplectic $8$-dimensional solvmanifold}\label{8-dim-solvmanifold}

\noindent
In this section we show a closed manifold of dimension 8 that is \emph{\texttt{c}}-flexible, \emph{\texttt{f}}-flexible and \emph{\texttt{h}}-flexible.  %%%This example first appeared in \cite{FIL1}, however some of the results presented
%%%in that paper for this solvmanifold need to be corrected as we show below.

Let us consider the $8$-dimensional compact solvmanifold $M=S/\Gamma$, where $S$ is a simply connected completely solvable Lie group of dimension $8$ defined by left-invariant $1$-forms $\{e^i,\, 1\leq i \leq 8\}$ such that
$$de^1=de^2=de^3=0,\quad de^4=-e^{12},\quad de^5=-e^{13},\quad de^6=-e^{14},\quad de^7=-e^{17},\quad de^8=e^{18}.$$
%%%It is easy to see that the manifold $M$ has the structure of an $S^1$-budle over a 7-dimensional manifold
%%%that is a $\mathbb{T}^2$-bundle over ${\rm Sol}(3)\times \mathbb{T}^2$, where ${\rm Sol}(3)$ is the solvable 3-dimensional manifold.
Using \cite{Hattori}, the Betti numbers of $M$ are
$b_1(M)=3$, $b_2(M)=7$, $b_3(M)=11$ and $b_4(M)=12$ (see \cite{FIL1} for a description of the de Rham cohomology groups of $M$).

The de Rham cohomology class of a generic symplectic $2$-form $\omega$ on $M$ is given by
$$
[\omega]=A\,[e^{15}]+B\,[e^{16}]+C\,[e^{23}]+D\,[e^{24}]+E\,[e^{25}+e^{34}]+F\,[e^{35}]+G\,[e^{78}],
$$
where the coefficients $A,\ldots,G\in \mathbb{R}$ satisfy $BG(E^2-DF)\neq 0$.
Notice that by Remark~\ref{coef-inv-cohom-class} we can suppose without loss of generality that for instance $G=1$.
Also notice that Propositions~\ref{calculo-coef-solvariedades} and~\ref{calculo-filtered-solvariedades}
allow us to reduce the computation of all the symplectic cohomology groups to the level of the Lie algebra $\mathfrak{s}$ of $S$.

The generalized coeffective cohomology groups for $k=3,4$ provide no flexibility since they are topological invariants.
A direct calculation shows that the \emph{\texttt{c}}-flexibility depends only on $\hat c_3$ and, moreover,
$\hat c_3=8$ or $9$ depending on $F\neq 0$ or $F=0$, respectively.

Next we show that $\hat c_3$ is also a key ingredient to obtain the other types of flexibility.
Using Theorem~\ref{rel_armonica_coef}, we have that $h_3-h_7=\hat c_3$.
It turns out that the map $L^3\colon H^1\longrightarrow H^7$ is identically zero, so $h_7=0$ and therefore $h_3=\hat c_3$.
On the other hand, according to Corollary~\ref{rel_armonica_coef-2-cor},
$\check c^{(2)}_6=b_6+h_3-h_6-h_7$, thus $\check c^{(2)}_6=h_3+b_2-h_6$, or equivalently,
$\check c^{(2)}_6= \hat c_3 + \text{dim}\ker \{L^2\colon H^2\longrightarrow H^6\}$.
A direct calculation shows that the dimension of the kernel of the latter $L^2$-map
is independent of the symplectic form and it is equal to 2, therefore $\check c^{(2)}_6= \hat c_3 + 2$.
Recall that $\check c^{(2)}_6$ is by Remark~\ref{comparison with filtered} the dimension
of the primitive cohomology group $PH^3_{d+d^\Lambda} \cong PH^3_{d d^\Lambda}$.

From this general study one concludes that the closed manifold $M$ is \emph{\texttt{c}}-flexible, \emph{\texttt{f}}-flexible and \emph{\texttt{h}}-flexible.
For instance, if we consider the following family of symplectic forms
$$
\omega_t= e^{15} + e^{16} + t\,e^{23} - t\,e^{24} + e^{25}+e^{34} +t\, e^{35} + e^{78} ,\ \quad t\in \mathbb{R},
$$
then:

\smallskip

\noindent $\bullet$ $\hat c_3(\omega_{0})=h_3(\omega_{0})=9$, and $\check c_6^{(2)}(\omega_{0})=11$;

\smallskip

\noindent $\bullet$ $\hat c_3(\omega_{t})=h_3(\omega_{t})=8$, and $\check c_6^{(2)}(\omega_{t})=10$,\, for any $t\neq 0$.

\medskip

\begin{remark}
{\rm
The symplectic manifold $(M,\omega_0)$ is considered in~\cite{FIL1} (see also \cite[page 288]{FIL2}) as an example of
a compact symplectic manifold of dimension 8 for which~\eqref{isotrun} does not hold.
However, the $1$-coeffective cohomology groups were wrongly obtained and the conclusion is not correct.
In fact, one can prove that~\eqref{isotrun} holds for any $t\in \mathbb{R}$, in particular for $t=0$.

\smallskip

On the other hand, by Remark~\ref{k-truncada}, for any symplectic manifold $\M$ satisfying the HLC
(in particular, for any compact K\"ahler manifold)
the $k$-coeffective cohomology group $H^{q}_{(k)}(M)$ is isomorphic to the $[\omega^k]$-truncated de Rham group
$\widetilde{H}^q_{[\omega^k]}(M)=\{[\alpha]\in H^q(M) \mid [\alpha]\cup[\omega^k]=0\}$,
for any $q\geq n-k+2$ and $1\leq k\leq n$.
It is a natural question if such isomorphisms hold for any $k$ for arbitrary symplectic manifolds.
A detailed study of the symplectic manifolds $(M,\omega_t)$ above allows us to conclude that

\smallskip

\noindent $\bullet$ $H^q_{(k)}(M,\omega_t)\cong \widetilde{H}^q_{[\omega_t^k]}(M)$ for $k=1,3,4$ and for any $q\geq 6-k$ and $t\in \mathbb{R}$;

\smallskip

\noindent $\bullet$ $H^q_{(2)}(M,\omega_t)\cong \widetilde{H}^q_{[\omega_t^2]}(M)$ for any $q\geq 4$ if and only if $t\neq 0$.

\medskip

\noindent Therefore, for the compact symplectic manifold $(M,\omega_0)$
we have that the 2-coeffective cohomology is not isomorphic to the $[\omega^2]$-truncated de Rham cohomology.
More precisely, one has that
$H^4_{(2)}(M,\omega_0) \not\cong \widetilde{H}^4_{[\omega_0^2]}(M)$.
}
\end{remark}

\section{Flexibility of symplectic $2$-step nilmanifolds}\label{Sakane-Yamada-sec}

In this section we find symplectic $2$-step nilmanifolds of arbitrary high dimension
which are \emph{\texttt{c}}-flexible, \emph{\texttt{f}}-flexible and \emph{\texttt{h}}-flexible.

\begin{proposition}\label{2-step-nilm}
Let $M=G/\Gamma$ be a $2$-step nilmanifold of dimension $2n \geq 6$, endowed with a continuous family of symplectic forms $\omega_t$.
%%%endowed with a (an invariant) symplectic form $\omega$.
Then, the harmonic number $h_3(M,\omega_t)$ varies if and only if the coeffective number $\hat c_3(M,\omega_t)$ varies,
if and only if the filtered number $\check c_{2n-1}^{(n-2)}(M,\omega_t)$ varies.
%%%Moreover, in such case $M$ is \emph{\texttt{f}}-flexible.
\end{proposition}

\begin{proof}
Let us denote by $\frg$ the Lie algebra of the Lie group $G$.
By \cite[Theorem 3]{Ym}, for any symplectic 2-step nilmanifold
$$
b_1-h_{2n-1}=\dim [\frg,\frg],
$$
which implies that the harmonic number $h_{2n-1}$ does not depend on the symplectic form.
From Theorem~\ref{rel_armonica_coef}
we get $h_{3} - h_{2n-1} = \hat c_{3}$. Therefore, $h_3$ varies if and only if $\hat c_3$ does.

%%Now, let us suppose that there is a continuous family of symplectic forms $\omega_t$ on $M$
%%such that $\hat c_3(M,\omega_t)$ varies.

On the other hand, by Definition~\ref{chi-k} and Proposition~\ref{invtopologicos} for $k=n-2$ we have
\begin{equation}\label{ecu-aux}
- \hat c_{3}+ \sum^{2n}_{i=4}(-1)^i \, c_i^{(n-2)}=\chi^{(n-2)}=\sum_{r=3}^{2n-2}(-1)^{r} \, b_{r}.
\end{equation}
Now, from Proposition~\ref{igualdades-general}~(ii) and Proposition~\ref{igualdades}~(i) for $k=n-2$
we get the following identities:
$$
c_{4}^{(n-2)} = \check{c}_{2n-1}^{(n-2)}, \quad \mbox{ and } \quad
c_i^{(n-2)} = b_i  \ \mbox{ for every }\  5\leq i\leq 2n.
$$
Therefore, equation \eqref{ecu-aux} reduces to
$$
- \hat c_{3} + \check{c}_{2n-1}^{(n-2)} + \sum^{2n-2}_{i=5}(-1)^i \, b_i - b_{2n-1}+b_{2n}=-b_3+b_4+\sum_{r=5}^{2n-2}(-1)^{r} \, b_{r},
$$
which implies
$$
\check{c}_{2n-1}^{(n-2)} = \hat c_{3} -b_3+b_4+b_{2n-1}-b_{2n}.
$$
Consequently, the filtered number $\check c_{2n-1}^{(n-2)}(M,\omega_t)$ varies if and only if
the coeffective number $\hat c_3(M,\omega_t)$ varies.
%%%
%%%relacion con la primitiva $PH$ ??? (NO)
%%%
\end{proof}

We apply Proposition~\ref{2-step-nilm} to
a family of examples found by Sakane and Yamada in \cite{SY-proc}. The conclusion is that for any $k \geq 2$,
there is a $6k$-dimensional symplectic nilmanifold which is \emph{\texttt{c}}-flexible, \emph{\texttt{f}}-flexible and \emph{\texttt{h}}-flexible.

\begin{example}\label{Sakane-Yamada}
{\rm
Let $k$ be an integer such that $k \geq 2$.
We consider the $6k$-dimensional compact nilmanifold $M=G/\Gamma$, where $G$ is a simply connected $2$-step nilpotent Lie group
%%%of dimension $6k$
defined by left-invariant $1$-forms $\{\alpha^i,\, \beta^i,\, 1\leq i \leq 3k\}$ satisfying
\begin{equation}\label{str-ecus-2step-nilm}
\left\{\begin{array}{l}
d\alpha^1=\cdots=d\alpha^{3k}=0, \\[4pt]
d\beta^1=\alpha^{1}\wedge\alpha^{2}, \\[4pt]
d\beta^2=\alpha^{2}\wedge\alpha^{3},\\[-1pt]
\ \ \, \vdots\ \\[4pt]
d\beta^{3k-1}=\alpha^{3k-1}\wedge\alpha^{3k}, \\[4pt]
d\beta^{3k}=\alpha^{3k}\wedge\alpha^{1}.
\end{array}\right.
\end{equation}

We denote by $\frg$ the Lie algebra of $G$. Let $\mathfrak{a}^0$ be a complementary vector subspace of the derived
algebra $[\frg,\frg]$ in $\frg$,
and consider $\mathfrak{a}^1=[\frg,\frg]$. So, $\frg=\mathfrak{a}^0\oplus\mathfrak{a}^1$ as a vector space,
and this decomposition induces a bigraduation
$
\bigwedge^r \frg^* = \oplus_{i_0+i_1=r} \bigwedge^{i_0} (\mathfrak{a}^0)^* \otimes \bigwedge^{i_1} (\mathfrak{a}^1)^*,
$
for any $r$. For simplicity, we denote $\bigwedge^{i_0,i_1}=\bigwedge^{i_0} (\mathfrak{a}^0)^* \otimes \bigwedge^{i_1} (\mathfrak{a}^1)^*$.
Hence, for $r=3$ we have
$$
\bigwedge\ ^{\hskip-.2cm 3} \ \frg^*
= \bigwedge\ ^{\hskip-.2cm 3,0} \oplus \bigwedge\ ^{\hskip-.2cm 2,1} \oplus\bigwedge\ ^{\hskip-.2cm 1,2} \oplus\bigwedge\ ^{\hskip-.2cm 0,3}.
$$
Let $Z^3(\frg)$ and $B^3(\frg)$ denote the subspaces of $\bigwedge^3 \frg^*$ consisting of closed and exact 3-forms,
respectively.

For any invariant symplectic form $\omega$ on $M$, we denote by $\mathcal{H}^3(\frg,\omega)$ the space of invariant $\omega$-harmonic
3-forms. Since $M$ is 2-step, by \cite[Theorem 3]{SY-proc} we have that $B^3(\frg) \subset \mathcal{H}^3(\frg,\omega)$.
Hence, the harmonic cohomology group $H^3_\hr(M,\omega)$ satisfies
$$
H^3_\hr(M,\omega)\cong \frac{\mathcal{H}^3(\frg,\omega)}{B^3(\frg) \cap \mathcal{H}^3(\frg,\omega)}\cong \frac{\mathcal{H}^3(\frg,\omega)}{B^3(\frg)}.
$$
Moreover, it is easy to see that
$$
\mathcal{H}^3(\frg,\omega)= \bigwedge\ ^{\hskip-.2cm 3,0} \oplus\, \left(Z^3(\frg) \cap \bigwedge\ ^{\hskip-.2cm 2,1}\right) \oplus\, \left(\mathcal{H}^3(\frg,\omega)\cap \bigwedge\ ^{\hskip-.2cm 1,2}\right),
$$
and
$$
B^3(\frg) \subset \bigwedge\ ^{\hskip-.2cm 3,0} \oplus\, \left(Z^3(\frg) \cap \bigwedge\ ^{\hskip-.2cm 2,1}\right).
$$
Hence, the harmonic number $h_3$
only depends on the dimension of the space $\mathcal{H}^3(\frg,\omega)\cap \bigwedge\ ^{\hskip-.2cm 1,2}$.

Now, let us consider the symplectic form
$\tau = a_{1}\, \alpha^1\wedge\beta^1 + \cdots + a_{3k}\, \alpha^{3k}\wedge\beta^{3k}$, where $a_{1},\ldots, a_{3k} \not=0$.
A direct calculation using \eqref{str-ecus-2step-nilm} shows that
$$
\dim \left(\mathcal{H}^3(\frg,\tau)\cap \bigwedge\ ^{\hskip-.2cm 1,2}\right)=0.
$$

Let $\sigma = b_{1}(\alpha^1\wedge\beta^2 - \alpha^3\wedge\beta^1)
+ \cdots +
b_{3k-2}(\alpha^{3k-2}\wedge\beta^{3k-1} - \alpha^{3k}\wedge\beta^{3k-2})
+ b_{3k-1}(\alpha^{3k-1}\wedge\beta^{3k} - \alpha^{1}\wedge\beta^{3k-1})$. One can choose $b_{1},\ldots, b_{3k-1}$ such that
$\sigma$ is non-degenerate (for more details see \cite{SY-proc}) and, since it is closed, $\sigma$ defines another symplectic form on $M$.
Using again the structure equations \eqref{str-ecus-2step-nilm} one can prove that
$\mathcal{H}^3(\frg,\sigma)\cap \bigwedge\ ^{\hskip-.2cm 1,2}=\langle \alpha^{j+1}\wedge\beta^{j}\wedge\beta^{j+1} \ ; \  j=1,\ldots,3k \rangle$,
which implies
$$
\dim \left(\mathcal{H}^3(\frg,\sigma)\cap \bigwedge\ ^{\hskip-.2cm 1,2}\right)=3k.
$$

We take $\varepsilon>0$ sufficiently small so that the closed 2-form $\omega_t\colon=\sigma+t\, \tau$ is non-degenerate
for any $t \in [0,\varepsilon]$.
The space $\mathcal{H}^3(\frg,\omega_t)\cap \bigwedge\ ^{\hskip-.2cm 1,2}$ has dimension $3k$ for $t=0$, and dimension $0$ for $t\not=0$.
This proves that the harmonic cohomology group $H^3_\hr(M,\omega_t)$ varies with $t$.
Furthermore, since the harmonic number $h_3(M,\omega_t)$ varies, by Proposition~\ref{2-step-nilm} we get that the coeffective number $\hat c_3(M,\omega_t)$
and the filtered number $\check c_{2n-1}^{(n-2)}(M,\omega_t)$ also vary with $t$.
}
\end{example}

\noindent{\textbf{Acknowledgments}}.
This work has been partially supported by the projects MTM2017-85649-P (AEI/FEDER, UE), and
E22-17R ``Algebra y Geometr\'{\i}a'' (Gobierno de Arag\'on/FEDER).

\end{document}